\documentclass[a4paper, 12pt]{article}

\usepackage[margin=25truemm]{geometry}
\usepackage{hyperref}

\title{\Large\bf 
A generalization of quantum Lakshmibai-Seshadri paths for arbitrary weights%
}
\author{%
Takafumi Kouno \\
 \small Waseda Research Institute for Science and Engineering, Waseda University, \\
 \small 3-4-1 Okubo, Shinjuku-ku, Tokyo 169-8555, Japan \\
 \small (e-mail: {\tt t.kouno@aoni.waseda.jp}) \\[5mm]
Satoshi Naito \\ 
 \small Department of Mathematics, Tokyo Institute of Technology, \\
 \small 2-12-1 Oh-okayama, Meguro-ku, Tokyo 152-8551, Japan \\
 \small (e-mail: {\tt naito@math.titech.ac.jp}) \\[5mm]
}
\date{}

\usepackage{amsmath, amssymb}
\usepackage{mathtools}
\mathtoolsset{showonlyrefs}
\usepackage{bm}
\usepackage{color}

\numberwithin{equation}{section}

\usepackage{amsthm}
\theoremstyle{plain}
\newtheorem{theorem}{Theorem}[section]
\newtheorem{itheorem}{Theorem}
\newtheorem*{theorem*}{Theorem}
\newtheorem{lemma}[theorem]{Lemma}
\newtheorem{proposition}[theorem]{Proposition}
\newtheorem{corollary}[theorem]{Corollary}

\theoremstyle{remark}
\newtheorem{remark}[theorem]{Remark}
\newtheorem{example}[theorem]{Example}

\theoremstyle{definition}
\newtheorem{definition}[theorem]{Definition}

\newcommand{\BZ}{\mathbb{Z}}
\newcommand{\BQ}{\mathbb{Q}}
\newcommand{\BR}{\mathbb{R}}
\newcommand{\BC}{\mathbb{C}}
\newcommand{\CA}{\mathcal{A}}

\newcommand{\CS}{\mathcal{S}}
\newcommand{\Fg}{\mathfrak{g}}
\newcommand{\Fh}{\mathfrak{h}}
\newcommand{\sL}{\mathsf{L}}
\newcommand{\q}{\mathsf{q}}

\newcommand{\bp}{\mathbf{p}}
\newcommand{\bchi}{\bm{\chi}}

\newcommand{\af}{\mathrm{af}}
\newcommand{\ex}{\mathrm{ex}}
\newcommand{\Waf}{W_{\af}}
\newcommand{\Wex}{W_{\ex}}

\newcommand{\vtl}{\vartriangleleft}
\newcommand{\vtr}{\vartriangleright}

\newcommand{\vpi}{\varpi}

\newcommand{\Dp}[1]{\Delta^{+}(#1)_{>0}}
\newcommand{\De}[1]{\Delta^{+}(#1)_{=0}}
\newcommand{\Dn}[1]{\Delta^{+}(#1)_{<0}}
\newcommand{\Daf}[1]{\Delta^{\vee, +}_{\af} \cap t_{#1}^{-1} \Delta^{\vee, -}_{\af}}

\newcommand{\wti}[1]{\widetilde{#1}}
\newcommand{\tXi}{\widetilde{\Xi}}
\newcommand{\RO}{\mathcal{RO}}
\newcommand{\Par}{\mathrm{Par}}
\newcommand{\bPar}{\overline{\Par}}

\newcommand{\bra}[1]{[\![ #1 ]\!]}
\newcommand{\pra}[1]{(\!( #1 )\!)}
\newcommand{\pair}[2]{\langle #1, #2 \rangle}

\DeclareMathOperator{\Hom}{Hom}
\DeclareMathOperator{\QBG}{QBG}
\DeclareMathOperator{\BG}{BG}
\DeclareMathOperator{\QLS}{QLS}
\DeclareMathOperator{\IQLS}{IQLS}
\DeclareMathOperator{\ILS}{ILS}
\DeclareMathOperator{\wt}{wt}
\DeclareMathOperator{\down}{down}
\DeclareMathOperator{\height}{height}
\DeclareMathOperator{\nega}{neg}
\DeclareMathOperator{\ed}{end}
\DeclareMathOperator{\im}{Im}

\DeclareMathOperator{\gch}{gch}
\DeclareMathOperator{\Deg}{Deg}
\DeclareMathOperator{\sgn}{sgn}

\newenvironment{enu}{%
 \begin{enumerate}%
}{\end{enumerate}}

\begin{document}

\maketitle

\begin{abstract}
We construct an injective weight-preserving map (called the forgetful map) from the set of all admissible subsets in the quantum alcove model associated to an arbitrary weight. The image of this forgetful map can be explicitly described by introducing the notion of ``interpolated quantum Lakshmibai-Seshadri (QLS for short) paths'', which can be thought of as a generalization of quantum Lakshmibai-Seshadri paths. 
As an application, we reformulate, in terms of interpolated QLS paths, an identity of Chevalley type for the graded characters of Demazure submodules of a level-zero extremal weight module over a quantum affine algebra, which is a representation-theoretic analog of the Chevalley formula for the torus-equivariant $K$-group of a semi-infinite flag manifold. 
\end{abstract}

\section{Introduction}

In \cite{LS}, the authors constructed a weight-preserving bijection from the set of all admissible subsets in the alcove model associated to the lex (lexicographic) $\lambda$-chain (of roots) onto the set of all Lakshmibai-Seshadri (LS for short) paths of shape $\lambda$ for a dominant weight $\lambda \in P^{+}$; such a bijection was constructed in order to reformulate the Chevalley formula for the torus-equivariant $K$-theory of a Kac-Moody thick flag manifold, originally stated in terms of LS paths, in terms of the alcove model. 

In \cite{LNSSS2}, the authors constructed a weight-preserving bijection (called the forgetful map) from the set of all admissible subsets in the quantum alcove model associated to the lex $\lambda$-chain onto the set of all quantum Lakshmibai-Seshadri (QLS for short) paths of shape $\lambda$ for a dominant weight $\lambda \in P^{+}$; this forgetful map was used in order to establish the equality between the specialization at $t = 0$ of the symmetric Macdonald polynomial $P_{\lambda}(x; q, t)$ and the graded character of a tensor product of ``single-column'' Kirillov-Reshetikhin modules over a quantum affine algebra.

Also, in \cite{NNS}, the authors considered the forgetful map, defined on the set of all admissible (but, with arrows reversed) subsets in the quantum alcove model associated to each reduced expression for the shortest element $m_{\lambda}$ in the coset $t_{\lambda} W \in \Wex^{\vee}/W$ for an arbitrary (not necessarily dominant) weight $\lambda \in P$, where $t_{\lambda}$ denotes the translation element of the extended affine Weyl group $\Wex^{\vee} = W \ltimes P$, with $W$ the finite Weyl group. This forgetful map is injective and weight-preserving, but its image is a subset (which is explicitly described in terms of the final direction of a QLS path) of the set of all QLS paths of shape $\lambda_{+}$, where $\lambda_{+} \in P^{+}$ is the unique dominant weight in the $W$-orbit of $\lambda$; note that if $\lambda = w_{\circ} \lambda_{+}$, with $w_{\circ} \in W$ the longest element, then $m_{\lambda} = t_{\lambda}$, and hence the subset above is identical to the whole of the set of all QLS paths of shape $\lambda_{+}$. 
By using this result, they obtained an explicit relationship between the specialization at $t = \infty$ of the nonsymmetric Macdonald polynomial $E_{w_{\circ} \lambda_{+}}(x; q, t)$ and the graded character of the Demazure submodule $V_{w_{\circ}}^{-}(\lambda_{+})$ of the level-zero extremal weight module $V(\lambda_{+})$ of extremal weight $\lambda_{+} \in P^{+}$ over the quantum affine algebra $U_{\q}(\Fg_{\af})$. 

In this paper, in order to reformulate (a representation-theoretic analog of) the general Chevalley formula, obtained in \cite{LNS}, for the torus-equivariant $K$-group of a semi-infinite flag manifold, we consider an analog of the forgetful map above, defined on the set of all admissible subsets in the quantum alcove model associated to a ``suitable'' (in the sense explained at the end of Section~\ref{subsec:reduced_expression_vs_chain_of_roots} below) reduced expression for the translation element $t_{\lambda} \in \Wex^{\vee}$ for an arbitrary weight $\lambda \in P$. Here we should remark that for a general weight $\lambda \in P$, the translation element $t_{\lambda}$ differs from the shortest element $m_{\lambda}$ in the coset $t_{\lambda} W$. 
Hence the image of our forgetful map (which is still injective and weight-preserving) cannot be described in terms of QLS paths, and we need to introduce the notion of \emph{interpolated QLS paths}, which can be thought of as a generalization of QLS paths (see Remarks~\ref{rem:IQLS=QLS} and \ref{rem:minuscule case}). To be a little more precise, an interpolated QLS path of shape $\lambda \in P$ is a triple consisting of 
\begin{itemize}
\item two sequences of elements of $W$, and 
\item an increasing sequence of rational numbers between $0$ and $1$
\end{itemize}
satisfying several conditions (which are similar to those in the definition of QLS paths); for the precise definition, see Definition~\ref{def:iQLS}. 
Here we remark that in contrast to QLS paths, we need \emph{two} sequences of elements in $W$, while a QLS path is a pair of \emph{one} sequence of elements in $W$ and a sequence of increasing rational numbers between $0$ and $1$. 
Let $\IQLS(\lambda)$ denote the set of all interpolated QLS paths of shape $\lambda \in P$. 

For an arbitrary weight $\lambda \in P$ and $w \in W$, let us take the suitable $\lambda$-chain $\Gamma_{\vtl}(\lambda)$, and let $\CA(w, \Gamma_{\vtl}(\lambda))$ denote the set of all $w$-admissible subsets in the quantum alcove model associated to $\Gamma_{\vtl}(\lambda)$. 

The main results of this paper are the following theorems.

\begin{itheorem}[= Definition~\ref{def:forgetful} + Theorem~\ref{thm:inj_forgetful} + Proposition~\ref{prop:forgetful_wt}]
Let $\lambda \in P$ be an arbitrary weight and $w \in W$. 
There exists an injective weight-preserving map $\tXi: \CA(w, \Gamma_{\vtl}(\lambda)) \rightarrow \IQLS(\lambda) \times W$. 
\end{itheorem}

Notice that in the theorem above, for the injectivity assertion, we need to modify the forgetful map by taking the direct product $\IQLS(\lambda) \times W$ instead of $\IQLS(\lambda)$ as its range. 

As for the image of the forgetful map $\tXi$, we have the following explicit description; $\iota(\eta) \in W$ and $\kappa(\eta) \in W$ are the initial and final directions of $\eta \in \IQLS(\lambda)$, respectively (see Definition~\ref{def:integrality} for the notation $\xRightarrow{(\lambda, +)}$ and $\xRightarrow{(\lambda, -)}$). 

\begin{itheorem}[= Theorem~\ref{thm:im_forgetful}]
Let $\lambda \in P$ be an arbitrary weight and $w \in W$. 
Then the following equality holds: 
\begin{equation}
\im(\tXi) = \left\{ (\eta, u) \in \IQLS(\lambda) \times W \ \middle| \begin{array}{ll} w \xRightarrow{(\lambda, +)} \kappa(\eta) \\ \iota(\eta) \xRightarrow{(\lambda, -)} u \end{array} \right\}. 
\end{equation}
\end{itheorem}

As an application of these results, we can reformulate, in terms of interpolated QLS paths, an \emph{identity of Chevalley type} for the graded characters of Demazure submodules of a level-zero extremal weight module over a quantum affine algebra; 
note that this identity is a representation-theoretic analog of the Chevalley formula for the torus-equivariant $K$-group of a semi-infinite flag manifold, given in \cite{LNS}. 

For $\mu \in P^{+}$, let $V(\mu)$ denote the extremal weight module of extremal weight $\mu$ over the quantum affine algebra $U_{\q}(\Fg_{\af})$ associated to the (untwisted) affine Lie algebra $\Fg_{\af}$ with the underlying finite-dimensional simple Lie algebra $\Fg$. Also, for an element $x$ of the affine Weyl group $\Waf = W \ltimes Q^{\vee}$ of $\Fg_{\af}$, with $Q^{\vee}$ the coroot lattice of $\Fg$, 
let $\gch V_{x}^{-}(\mu)$ denote the graded character of the Demazure submodule $V_{x}^{-}(\mu) \subset V(\mu)$ corresponding to $x$; for the precise notation and definitions, see Sections~\ref{sec:forgetful} and \ref{sec:equivariant}. 

\begin{itheorem}[= Theorem~\ref{thm:gch_Chevalley}]
Let $x \in \Waf$, and write it 
as $x = w t_{\xi}$, with $w \in W$ and $\xi \in Q^{\vee}$. 
Let $\mu \in P^{+}$ and $\lambda \in P$ be such that $\mu + \lambda \in P^{+}$. 
Then we have the following equality:
\begin{equation}
\begin{split}
& \gch V_{x}^{-}(\mu + \lambda) =\\ 
& \sum_{\substack{\eta \in \IQLS(\lambda) \\ w \xRightarrow{(\lambda, +)} \kappa(\eta)}} \sum_{\substack{u \in W \\ \iota(\eta) \xRightarrow{(\lambda, -)} u}} \sum_{\bchi \in \bPar(\lambda)} (-1)^{\nega(\eta) + \ell(u) - \ell(\iota(\eta))} q^{\Deg_{w}(\eta) - \pair{\lambda}{\xi} - |\bchi|} \\ 
& \hspace*{60mm} \times e^{\wt(\eta)} \gch V_{ut_{\xi + \xi(u, \eta, w) + \iota(\bchi)}}^{-}(\mu). 
\end{split}
\end{equation}
\end{itheorem}

One advantage of the QLS path model over the quantum alcove model is that its canonical (affine) crystal structure is much easier to describe than that of the quantum alcove model;
compare \cite[\S 2.3]{LNSSS2} with \cite{LL1} (see also \cite[\S 2.4]{LL2}). 
Hence we expect that it would be possible to equip the set $\IQLS(\lambda)$ for an arbitrary weight $\lambda \in P$ with a ``signed'' crystal structure, and obtain a fairly simple description of it similar to that for QLS paths (cf.~\cite[\S 5.4]{KLN}). 

This paper is organized as follows. In Section~\ref{sec:notation}, we fix basic notation for root systems and recall the definition of the quantum Bruhat graph. Also, we define some specific reflection orders on the set of positive roots. In Section~\ref{sec:quantum_alcove_model}, we review the quantum alcove model. In Section~\ref{sec:def_IQLS}, we introduce the notion of interpolated QLS paths. In Section~\ref{sec:forgetful}, we construct the forgetful map, which is injective and preserves weights. In Section~\ref{sec:equivariant}, we reformulate, in terms of interpolated QLS paths, the identity of Chevalley type for the graded characters of level-zero Demazure submodules. 

The main results of this paper form a part of the doctoral thesis \cite{Kou} of T.K. 

\subsection*{Acknowledgments}
The authors would like to thank Cristian Lenart, Fumihiko Nomoto, and Daisuke Sagaki for valuable discussions on the definition of interpolated quantum Lakshmibai-Seshadri paths. T.K. was partly supported by JSPS Grants-in-Aid for Scientific Research 20J12058, 22J00874, 22KJ2908. S.N. was partly supported by JSPS Grant-in-Aid for Scientific Research (C) 21K03198.

\section{Basic notation and definitions}\label{sec:notation}

First we fix basic notation for root systems and affine root systems. 
Then we recall  the definition of the \emph{quantum Bruhat graph}, introduced in \cite{BFP}. 
Also, we define some specific reflection orders on $\Delta^{+}$. 

\subsection{Notation for root systems}

Let $\Fg$ be a complex simple Lie algebra with Cartan subalgebra $\Fh \subset \Fg$. 
We denote by $\pair{\cdot}{\cdot}$ the canonical pairing of $\Fh$ and $\Fh^{\ast} := \Hom_{\BC}(\Fh, \BC)$. 
Let $\Delta$ be the set of roots of $\Fg$, with $\Delta^{+} \subset \Delta$ the set of positive roots, 
and $\{\alpha_{i}\}_{i \in I}$ the set of simple roots; we denote by $\theta \in \Delta^{+}$ the highest root, and by $\vartheta \in \Delta^{+}$ the highest short root, and set 
$\rho := (1/2) \sum_{\alpha \in \Delta^{+}} \alpha$. 
For $\alpha \in \Delta^{+}$, we define $\sgn(\alpha) \in \{1, -1\}$ and $|\alpha| \in \Delta^{+}$ by 
\begin{align}
\sgn(\alpha) &:= \begin{cases} 1 & \text{if $\alpha \in \Delta^{+}$}, \\ -1 & \text{if $\alpha \in -\Delta^{+}$}, \end{cases} \\ 
|\alpha| &:= \sgn(\alpha) \alpha. 
\end{align}

We set $Q := \sum_{i \in I} \BZ \alpha_{i}$, 
$Q^{\vee} := \sum_{i \in I} \BZ \alpha_{i}^{\vee}$, 
where $\alpha^{\vee}$ is the coroot of $\alpha \in \Delta$, and 
$Q^{\vee, +} := \sum_{i \in I} \BZ_{\ge 0} \alpha_{i}^{\vee}$. 
Let $W = \langle s_{i} \mid i \in I \rangle$ be the Weyl group of $\Fg$, 
with $\ell: W \rightarrow \BZ_{\geq 0}$ the length function and 
$w_{\circ} \in W$ the longest element; here, for $\alpha \in \Delta^{+}$, $s_{\alpha} \in W$ denotes the reflection 
corresponding to $\alpha$, and $s_{i} = s_{\alpha_{i}}$ is the simple reflection for $i \in I$. 
Let $P := \sum_{i \in I} \BZ \vpi_{i}$ be the weight lattice of $\Fg$ and 
$P^{+} := \sum_{i \in I} \BZ_{\ge 0} \vpi_{i}$ the set of dominant weights, 
where $\vpi_{i}$ is the $i$-th fundamental weight for $i \in I$; 
also, we set $\Fh^{\ast}_{\BR} := P \otimes_{\BZ} \BR$.

\subsection{Affine root systems}

Let $\Fg_{\af} := \Fg \otimes_{\BC} \BC[t, t^{-1}] \oplus \BC c \oplus \BC d$ 
be the untwisted affine Lie algebra associated to $\Fg$, 
with Cartan subalgebra $\Fh_{\af} := \Fh \oplus \BC c \oplus \BC d$. 
Let $\Delta_{\af}$ be the set of roots of $\Fg_{\af}$, 
and $\Delta^{+}_{\af} \subset \Delta_{\af}$ the set of positive roots; 
we denote by $\delta \in \Delta_{\af}$ the (primitive) null root. 
We set $I_{\af} := I \sqcup \{ 0 \}$. 
Let $\{ \alpha_{i} \}_{i \in I_{\af}} \subset \Delta_{\af}^{+}$ be the set of simple roots, where $\alpha_{i}$, $i \in I$, are thought of as the simple roots of $\Fg$, and $\alpha_{0} := \delta - \theta$. 
Let $\Waf = \langle s_{i} \mid i \in I_{\af} \rangle$ denote the (affine) Weyl group of $\Fg_{\af}$, 
where $s_{i}$ for $i \in I$ are the simple reflections corresponding to $\alpha_{i}$, and $s_{0} = s_{\theta} t_{- \theta^{\vee}}$; 
we know (cf. \cite[Section~6.5]{Kac}) that $\Waf \cong \{t_{\xi} \mid \xi \in Q^{\vee}\} \rtimes W \simeq Q^{\vee} \rtimes W$, 
where $t_{\xi}$, $\xi \in Q^{\vee}$, denotes the translation by $\xi$. 

Also, let $\Fg^{\vee}$ be the (Langlans) dual Lie algebra of $\Fg$, 
and $\Fg_{\af}^{\vee}$ the untwisted affine Lie algebra associated to 
$\Fg^{\vee}$ with Cartan subalgebra $\Fh_{\af}^{\vee}$. 
Let $\Delta_{\af}^{\vee}$ be the set of roots of $\Fg_{\af}^{\vee}$ and 
$\Delta_{\af}^{\vee, +} \subset \Delta_{\af}^{\vee}$ the set of positive roots, 
with $\Delta_{\af}^{\vee, -} := - \Delta_{\af}^{\vee, +}$ the set of negative roots; 
we denote by $\widetilde{\delta} \in \Delta_{\af}^{\vee}$ the (primitive) null root. 
Each $\beta^{\vee} \in \Delta_{\af}^{\vee}$ is uniquely written as: $\beta^{\vee} = \gamma^{\vee} + k \widetilde{\delta}$ for $\gamma \in \Delta$ and $k \in \BZ$; 
we set $\overline{\beta^{\vee}} := \gamma^{\vee}$, and $\deg(\beta^{\vee}) := k$. 
Let $\Waf^{\vee} = \langle s_{i} \mid i \in I_{\af} \rangle$ denote the (affine) Weyl group of $\Fg_{\af}^{\vee}$, where 
by abuse of notation, we also denote by $s_{i}$ for $i \in I$ the simple reflections corresponding to $\alpha_{i}^{\vee}$, and $s_{0} = s_{\vartheta} t_{- \vartheta}$ is the simple reflection corresponding to the simple coroot $\alpha_{0}^{\vee} := \widetilde{\delta} - \vartheta^{\vee}$; 
we know (cf. \cite[Section~6.5]{Kac}) that $\Waf^{\vee} \cong \{ t_{\xi} \mid \xi \in Q \} \rtimes W \simeq Q \rtimes W$, 
where $t_{\xi}$, $\xi \in Q$, denotes the translation by $\xi$. 
In addition, let $\Wex^{\vee} \cong P \rtimes W$ be the extended affine Weyl group, where 
$t_{\mu} \in \Wex^{\vee}$ denotes the translation by $\mu \in P$. 
We denote by $\Omega^{\vee} \subset \Wex^{\vee}$ the (abelian) subgroup consisting of the elements of length zero; each element of $\Omega^{\vee}$ induces a graph automorphism of the Dynkin diagram of $\Fg_{\af}^{\vee}$. 
Note that we have $\Omega^{\vee} \cong P/Q \cong \Wex^{\vee}/\Waf^{\vee}$ and $\Wex^{\vee} \cong \Omega^{\vee} \ltimes \Waf^{\vee}$. 

\subsection{The quantum Bruhat graph}

\begin{definition}[{\cite[Definition~6.1]{BFP}}]
The \emph{quantum Bruhat graph} $\QBG(W)$ is the $\Delta^{+}$-labeled directed graph 
whose vertices are the elements of $W$ and whose edges are of the following form: 
$x \xrightarrow{\alpha} y$, with $x, y \in W$ and $\alpha \in \Delta^{+}$, such that $y = xs_{\alpha}$, 
and either of the following (B) or (Q) holds: 
\begin{itemize}
\item[(B)] $\ell(y) = \ell(x) + 1$; 
\item[(Q)] $\ell(y) = \ell(x) -2\pair{\rho}{\alpha^{\vee}} + 1$. 
\end{itemize}
If (B) (resp., (Q)) holds, then the edge $x \xrightarrow{\alpha} y$ is called a \emph{Bruhat edge} (resp., \emph{quantum edge}). 
\end{definition}

\begin{definition}[{cf. \cite[Section~2.1]{BB}}]
The \emph{Bruhat graph} $\BG(W)$ is the $\Delta^{+}$-labeled directed graph 
whose vertices are the elements of $W$ and whose edges are of the form $x \xrightarrow{\alpha} y$, with $x, y \in W$ 
and $\alpha \in \Delta^{+}$, such that $y = xs_{\alpha}$ and $\ell(y) = \ell(x) + 1$. 
Namely, the Bruhat graph $\BG(W)$ is the full subgraph of $\QBG(W)$ having only Bruhat edges. 
\end{definition}

Let $\bp: w_{0} \xrightarrow{\beta_{1}} w_{1} \xrightarrow{\beta_{2}} \cdots \xrightarrow{\beta_{r}} w_{r}$ be a directed path in $\QBG(W)$. 
We set 
\begin{align}
\ell(\bp) &:= r, \\ 
\ed(\bp) &:= w_{r}, \\ 
\wt(\bp) &:= \sum_{\substack{k \in \{1, \ldots, r\} \\ w_{k-1} \xrightarrow{\beta_{k}} w_{k} \text{ is a quantum edge}}} \beta_{k}^{\vee}.  
\end{align}

\begin{definition}[{cf. \cite[(2.2)]{D}}]
Let $w \in W$. 
A total order $\vtl$ on $w \Delta^{+}$ is a \emph{reflection order} 
if for all $\alpha, \beta \in w \Delta^{+}$ such that $\alpha + \beta \in w \Delta^{+}$, 
either $\alpha \vtl \alpha+\beta \vtl \beta$ or $\beta \vtl \alpha+\beta \vtl \alpha$ holds. 
\end{definition}

\begin{remark}
Let $w \in W$. If $\vtl$ is a reflection order on $w \Delta^{+}$, then the total order $\vtl^{\ast}$, 
defined by $\alpha \vtl^{\ast} \beta \Leftrightarrow \beta \vtl \alpha$, is also a reflection order on $w \Delta^{+}$. 
\end{remark}

We know from \cite[p.~662, Theorem]{Papi} that the set of reflection orders on $\Delta^{+}$ is in one-to-one correspondence 
with the set of reduced expressions for the longest element $w_{\circ}$ of $W$. 
More precisely, 
for each reduced expression $w_{\circ} = s_{i_{1}} \cdots s_{i_{r}}$, 
the total order $\vtl$ on $\Delta^{+}$, defined by $\alpha_{i_{1}} \vtl s_{i_{1}} \alpha_{i_{2}} \vtl \cdots \vtl s_{i_{1}} \cdots s_{i_{r-1}} \alpha_{i_{r}}$, 
is a reflection order. 
The assignment $(i_{1}, \ldots i_{r}) \mapsto {\vtl}$ gives a desired bijection. 
Using this fact, we can verify the following. 
%
%%%%%%%%%%%%%%%%%%%%%
% lem:reflection_coroot %
%%%%%%%%%%%%%%%%%%%%%
%
\begin{lemma}\label{lem:reflection_coroot}
Take a reflection order $\vtl$ on $\Delta^{+}$. 
Let $\alpha, \beta \in \Delta^{+}$ be such that $\alpha \vtl \beta$. 
If $\alpha^{\vee} + \beta^{\vee} \in \Delta^{\vee, +} = \{ \gamma^{\vee} \mid \gamma \in \Delta^{+} \}$, 
then $\alpha \vtl (\alpha^{\vee} + \beta^{\vee})^{\vee} \vtl \beta$. 
\end{lemma}

We can easily generalize Lemma~\ref{lem:reflection_coroot} to the one for reflection orders $\prec$ on $w\Delta^{+}$ for an arbitrary $w \in W$ as follows. 
%
%%%%%%%%%%%%%%%%%%%%%%%%%%%%
% cor:reflection_coroot_arbitrary %
%%%%%%%%%%%%%%%%%%%%%%%%%%%%
%
\begin{corollary}\label{cor:reflection_coroot_arbitrary}
Let $w \in W$. Take a reflection order $\prec$ on $w\Delta^{+}$. 
Let $\alpha, \beta \in w\Delta^{+}$ be such that $\alpha \prec \beta$. 
If $\alpha^{\vee} + \beta^{\vee} \in w\Delta^{\vee, +}$, 
then one has $\alpha \prec (\alpha^{\vee} + \beta^{\vee})^{\vee} \prec \beta$. 
\end{corollary}

\begin{proof}
We define a total order $\vtl$ on $\Delta^{+}$ by: $\alpha \vtl \beta \Leftrightarrow w\alpha \prec w\beta$ for $\alpha, \beta \in \Delta^{+}$. 
Then $\vtl$ is a reflection order on $\Delta^{+}$. 
Let $\alpha, \beta \in \Delta^{+}$ be such that $w\alpha \prec w\beta$, and assume that $(w\alpha)^{\vee} + (w\beta)^{\vee} \in w\Delta^{\vee, +}$. 
Then we have $\alpha \vtl \beta$, and $\alpha^{\vee} + \beta^{\vee} = w^{-1}((w\alpha)^{\vee} + (w\beta)^{\vee}) \in \Delta^{\vee, +}$. 
Hence Lemma~\ref{lem:reflection_coroot} implies that $\alpha \vtl (\alpha^{\vee} + \beta^{\vee})^{\vee} \vtl \beta$. 
Therefore, we deduce that $w\alpha \prec ((w\alpha)^{\vee} + (w\beta)^{\vee})^{\vee} \prec w\beta$, as desired. 
This proves the corollary. 
\end{proof}

Let $\vtl$ be a reflection order on $\Delta^{+}$. 
A directed path $\bp$ in $\QBG(W)$ of the form: 
\begin{equation}
\bp: w_{0} \xrightarrow{\beta_{1}} w_{1} \xrightarrow{\beta_{2}} \cdots \xrightarrow{\beta_{r}} w_{r},
\end{equation}
with $\beta_{1} \vtl \cdots \vtl \beta_{r}$, is called a \emph{label-increasing} directed path with respect to $\vtl$. 
%
%%%%%%%%%%%%%%%%
% thm:shellability %
%%%%%%%%%%%%%%%%
%
\begin{theorem}[{\cite[Theorem~6.4]{BFP}}]\label{thm:shellability}
Let $\vtl$ be a reflection order on $\Delta^{+}$. 
For all $v, w \in W$, there exists a unique label-increasing directed path from $v$ to $w$ in $\QBG(W)$ with respect to $\vtl$. 
Moreover, the unique label-increasing directed path from $v$ to $w$ has the minimum length. 
\end{theorem}

The assertion of Theorem~\ref{thm:shellability} is called the \emph{shellability property} of $\QBG(W)$. 
 
For all $v, w \in W$, 
there exists at least one shortest directed path $\bp$ in $\QBG(W)$ from $v$ to $w$; 
we set 
\begin{equation}
\ell(v \Rightarrow w) := \ell(\bp), \quad \wt(v \Rightarrow w) := \wt(\bp). 
\end{equation}
Note that by \cite[Lemma~1\,(2)]{P} (or \cite[Proposition~8.1]{LNSSS1}), $\wt(v \Rightarrow w)$ is well-defined. 
Also, note that $\ell(v \Rightarrow w) \equiv \ell(w) - \ell(v) \mod 2$ for all $v, w \in W$. 

\subsection{Some specific reflection orders}

For $\lambda \in P$, we set 
\begin{align}
\Dp{\lambda} &:= \{ \alpha \in \Delta^{+} \mid \pair{\lambda}{\alpha^{\vee}} > 0 \}, \\
\De{\lambda} &:= \{ \alpha \in \Delta^{+} \mid \pair{\lambda}{\alpha^{\vee}} = 0 \}, \\
\Dn{\lambda} &:= \{ \alpha \in \Delta^{+} \mid \pair{\lambda}{\alpha^{\vee}} < 0 \}; 
\end{align}
let $\RO(\lambda, \Delta^{+})$ denotes 
the set of all reflection orders $\vtl$ on $\Delta^{+}$ 
such that $\alpha \vtl \beta \vtl \gamma$ for all 
$\alpha \in \Dn{\lambda}$, $\beta \in \De{\lambda}$, and $\gamma \in \Dp{\lambda}$. 

\begin{remark}
The set $\RO(\lambda, \Delta^{+})$ is nonempty for each $\lambda \in P$. 
We will show this fact in Remark~\ref{rem:nonempty_RO}. 
\end{remark}

For $\lambda \in P$, there exists a unique $\lambda_{+} \in W \lambda$ such that $\lambda_{+} \in P^{+}$. 
Also, the set $\{ w \in W \mid w \lambda_{+} = \lambda \}$ has a unique maximal element (with respect to the Bruhat order), 
which we denote by $w(\lambda)$. 

\begin{lemma}
For $\lambda \in P$, we have $w(\lambda) \Delta^{+} = (\Dp{\lambda}) \sqcup (-\De{\lambda}) \sqcup (-\Dn{\lambda})$.
\end{lemma}
\begin{proof}
We set $I_{\lambda_{+}} := \{ i \in I \mid \pair{\lambda_{+}}{\alpha_{i}^{\vee}} = 0 \}$, and $\Delta_{\lambda_{+}}^{+} := \Delta^{+} \cap \sum_{i \in I_{\lambda_{+}}} \BZ_{\ge 0} \alpha_{i}$. 

If $\beta \in \Dp{\lambda}$, then $\pair{\lambda}{\beta^{\vee}} > 0$. Hence we have 
\begin{align}
\pair{\lambda_{+}}{w(\lambda)^{-1}\beta^{\vee}} = \pair{w(\lambda) \lambda_{+}}{\beta^{\vee}} = \pair{\lambda}{\beta^{\vee}} > 0. 
\end{align}
Since $\lambda_{+}$ is dominant, this implies that $w(\lambda)^{-1} \beta \in \Delta^{+} \setminus \Delta^{+}_{\lambda_{+}}$, and hence 
$\beta \in w(\lambda) \Delta^{+} \setminus w(\lambda) \Delta^{+}_{\lambda_{+}}$. 
If $\beta \in \Dn{\lambda}$, then $\pair{\lambda}{\beta^{\vee}} < 0$. Hence we have 
\begin{align}
\pair{\lambda_{+}}{w(\lambda)^{-1}\beta^{\vee}} = \pair{w(\lambda) \lambda_{+}}{\beta^{\vee}} = \pair{\lambda}{\beta^{\vee}} < 0. 
\end{align}
This implies that $w(\lambda)^{-1}\beta \in -(\Delta^{+} \setminus \Delta^{+}_{\lambda_{+}})$, and hence 
$-\beta \in w(\lambda) \Delta^{+} \setminus w(\lambda) \Delta^{+}_{\lambda_{+}}$. 
Thus, we have shown that $(\Dp{\lambda}) \sqcup (-\Dn{\lambda}) \subset w(\lambda) \Delta^{+} \setminus w(\lambda) \Delta^{+}_{\lambda_{+}}$. 

Next, we consider the relation between $\De{\lambda}$ and $w(\lambda)\Delta^{+}_{\lambda_{+}}$. 
We claim that $w(\lambda) \Delta^{+}_{\lambda_{+}} \subset - \Delta^{+}$. 
Let $i \in I_{\lambda_{+}}$. Since $\pair{\lambda_{+}}{\alpha_{i}^{\vee}} = 0$, or equivalently, $s_{i} \lambda_{+} = \lambda_{+}$, 
we have $w(\lambda) s_{i} \lambda_{+} = w(\lambda) \lambda_{+} = \lambda$. 
By the maximality of $w(\lambda)$, it follows that $w(\lambda) s_{i} < w(\lambda)$. 
This implies that $w(\lambda) \alpha_{i} \in -\Delta^{+}$. 
Thus we obtain $w(\lambda) \Delta^{+}_{\lambda_{+}} \subset - \Delta^{+}$, as desired. 
Let $\beta \in \Delta^{+}_{\lambda_{+}}$. 
Since $\pair{\lambda_{+}}{w(\lambda)^{-1}\beta^{\vee}} = \pair{\lambda}{\beta^{\vee}}$, 
we see that $\beta \in w(\lambda) \Delta^{+}_{\lambda_{+}}$ if and only if $\beta \in -\De{\lambda}$. 
Hence we obtain $w(\lambda) \Delta^{+}_{\lambda_{+}} = -\De{\lambda}$. 

From the above, we have $(\Dp{\lambda}) \sqcup (-\De{\lambda}) \sqcup (-\Dn{\lambda}) \subset w(\lambda) \Delta^{+}$. 
Since 
\begin{equation}
|\Dp{\lambda}| + |-\De{\lambda}| + |-\Dn{\lambda}| = |\Delta^{+}| = |w(\lambda) \Delta^{+}| (< \infty), 
\end{equation} 
we conclude that $(\Dp{\lambda}) \sqcup (-\De{\lambda}) \sqcup (-\Dn{\lambda}) = w(\lambda) \Delta^{+}$, as desired. 
This proves the lemma. 
\end{proof}

Based on this lemma, we consider the set of reflection orders on $w(\lambda) \Delta^{+}$ 
satisfying certain additional conditions. 

For $\lambda \in P$, let $\RO(\lambda, w(\lambda)\Delta^{+})$ denotes 
the set of all reflection orders $\prec$ on $w(\lambda) \Delta^{+}$ 
such that $\gamma \prec -\alpha \prec -\beta$ for all 
$\gamma \in \Dp{\lambda}$, $\alpha \in \Dn{\lambda}$, and $\beta \in \De{\lambda}$. 

\begin{example}
Assume that $\Fg$ is of type $A_{3}$. Let $\lambda = -\vpi_{1}+\vpi_{3}$. Then we see that $\lambda^{+} = \vpi_{2}$, $w(\lambda) = s_{1}s_{2}s_{1}s_{3}$, and 
\begin{equation}
w(\lambda)\Delta^{+} = \{ \underbrace{\alpha_{2}+\alpha_{3}, \alpha_{3}}_{\Delta^{+}(\lambda)_{>0}}, \underbrace{-\alpha_{1}, -\alpha_{1}-\alpha_{2}}_{-\Delta^{+}(\lambda)_{<0}}, \underbrace{-\alpha_{1}-\alpha_{2}-\alpha_{3}, -\alpha_{2}}_{-\Delta^{+}(\lambda)_{=0}} \}. 
\end{equation}
If we define a total order $\prec$ on $w(\lambda)\Delta^{+}$ by: 
\begin{equation}
\alpha_{2}+\alpha_{3} \prec \alpha_{3} \prec -\alpha_{1} \prec -\alpha_{1}-\alpha_{2} \prec -\alpha_{1}-\alpha_{2}-\alpha_{3} \prec -\alpha_{2}, 
\end{equation}
then we can verify that $\prec \, \in \RO(\lambda, w(\lambda)\Delta^{+})$. 
\end{example}

\begin{remark}
We can construct an element of $\RO(\lambda, w(\lambda)\Delta^{+})$ explicitly as follows. 
Let $u(\lambda) \in W$ be the unique minimal element of the set $\{w \in W \mid w\lambda_{+} = \lambda\}$. 
Also, we denote by $w_{\circ}(I_{\lambda_{+}})$ the longest element of the parabolic subgroup $W_{I_{\lambda_{+}}} := \langle s_{i} \mid i \in I_{\lambda_{+}} \rangle$ of $W$. 
Then we have the following length-additive decomposition: 
\begin{equation}
w_{\circ} = u(-\lambda)^{-1} u(\lambda) w_{\circ}(I_{\lambda_{+}}). 
\end{equation}
Hence, by taking reduced expressions 
\begin{align}
u(-\lambda) &= s_{i_{0}} s_{i_{-1}} \cdots s_{i_{-p}}, \\ 
u(\lambda) &= s_{i_{1}} s_{i_{2}} \cdots s_{i_{q}}, \\ 
w_{\circ}(I_{\lambda_{+}}) &= s_{i_{q+1}} s_{i_{q+2}} \cdots s_{i_{r}}, 
\end{align}
we obtain the following reduced expression: 
\begin{equation}
w_{\circ} = \underbrace{s_{i_{-p}} \cdots s_{i_{-1}} s_{i_{0}}}_{u(-\lambda)^{-1}} \underbrace{s_{i_{1}} s_{i_{2}} \cdots s_{i_{q}}}_{u(\lambda)} \underbrace{s_{i_{q+1}} s_{i_{q+2}} \cdots s_{i_{r}}}_{w_{\circ}(I_{\lambda_{+}})}. 
\end{equation}
Therefore, if we set 
\begin{equation}
\beta_{k} := s_{i_{r}} s_{i_{r-1}} \cdots s_{i_{k+1}} \alpha_{i_{k}}
\end{equation}
for $-p \le k \le r$, then the total order $\vtl$ on $\Delta^{+}$ defined by: $\beta_{-p} \vtl \beta_{-p+1} \vtl \cdots \vtl \beta_{r-1} \vtl \beta_{r}$ is a reflection order,
and hence the total order $\prec$ on $w(\lambda)\Delta^{+}$ defined by: $w(\lambda)\beta_{-p} \prec w(\lambda)\beta_{-p+1} \prec \cdots \prec w(\lambda)\beta_{r-1} \prec w(\lambda)\beta_{r}$ is a reflection order on $w(\lambda)\Delta^{+}$. 
By direct computation, we see that 
\begin{equation}
w(\lambda)\beta_{k} = \begin{cases}
s_{i_{0}} s_{i_{-1}} \cdots s_{i_{k+1}} \alpha_{i_{k}} & \text{if $-p \le k \le 0$}, \\ 
-s_{i_{1}} s_{i_{2}} \cdots s_{i_{k-1}} \alpha_{i_{k}} & \text{if $1 \le k \le q$}, \\ 
-u(\lambda) s_{i_{q+1}} s_{i_{q+2}} \cdots s_{i_{k-1}} \alpha_{i_{k}} & \text{if $q+1\le k \le r$}. 
\end{cases}
\end{equation}
Hence we have 
\begin{align}
\{ w(\lambda)\beta_{k} \mid -p \le k \le 0 \} &= \Delta^{+} \cap u(-\lambda)\Delta^{-} = \Dp{\lambda}, \\ 
\{ w(\lambda)\beta_{k} \mid 1 \le k \le q \} &= -(\Delta^{+} \cap u(\lambda) \Delta^{-}) = -\Dn{\lambda}, \\ 
\{ w(\lambda)\beta_{k} \mid q+1 \le k \le r \} &= -u(\lambda)\Delta_{\lambda_{+}}^{+}. 
\end{align}
From these, we conclude that ${\prec} \in \RO(\lambda, w(\lambda)\Delta^{+})$. 
\end{remark}

Let ${\vtl} \in \RO(\lambda, \Delta^{+})$. We write the elements of $\Delta^{+}$ as:  
\begin{equation}
\underbrace{\gamma_{-r} \vtl \cdots \vtl \gamma_{-q-2} \vtl \gamma_{-q-1}}_{\Dn{\lambda}} 
\vtl \underbrace{\gamma_{-q} \vtl \cdots \vtl \gamma_{-1} \vtl \gamma_{0}}_{\De{\lambda}} 
\vtl \underbrace{\gamma_{1} \vtl \gamma_{2} \vtl \cdots \vtl \gamma_{p}}_{\Dp{\lambda}}, 
\end{equation}
where $\Dn{\lambda} = \{ \gamma_{-r}, \ldots, \gamma_{-q-2}, \gamma_{-q-1} \}$, 
$\De{\lambda} = \{ \gamma_{-q}, \ldots, \gamma_{-1}, \gamma_{0} \}$, 
and $\Dp{\lambda} = \{ \gamma_{1}, \gamma_{2}, \ldots, \gamma_{p} \}$. 
Then, we define a total order ${\prec} = {\prec_{\vtl}}$ on $w(\lambda) \Delta^{+}$ by: 
\begin{equation}
\begin{split}
& \underbrace{\gamma_{1} \prec \gamma_{2} \prec \cdots \prec \gamma_{p}}_{\Dp{\lambda}} 
\prec \underbrace{-\gamma_{-r} \prec \cdots \prec -\gamma_{-q-2} \prec -\gamma_{-q-1}}_{-\Dn{\lambda}} \\ 
& \hspace{60mm} \prec \underbrace{-\gamma_{-q} \prec \cdots \prec -\gamma_{-1} \prec -\gamma_{0}}_{-\De{\lambda}}. 
\end{split} \label{eq:def_RO_2} %eq:def_RO_2
\end{equation}

%%%%%%%%%%%%%
% prop:bij_RO %
%%%%%%%%%%%%%

\begin{proposition} \label{prop:bij_RO}
The assignment ${\vtl} \mapsto {\prec_{\vtl}}$ gives a bijection between $\RO(\lambda, \Delta^{+})$ and $\RO(\lambda, w(\lambda) \Delta^{+})$. 
\end{proposition}
\begin{proof}
We show that ${\prec} = {\prec_{\vtl}}$ is a reflection order on $w(\lambda) \Delta^{+}$. 
Let $\alpha, \beta \in w(\lambda) \Delta^{+}$ be such that $\alpha \prec \beta$ and $\alpha + \beta \in w(\lambda) \Delta^{+}$. 
Then $\alpha$ and $\beta$ satisfy one of the following:
\begin{enu}
\item $\alpha, \beta \in \Dp{\lambda}$; 
\item $\alpha, \beta \in -\De{\lambda}$; 
\item $\alpha, \beta \in -\Dn{\lambda}$; 
\item $\alpha \in \Dp{\lambda}$ and $\beta \in -\Dn{\lambda}$; 
\item $\alpha \in \Dp{\lambda}$ and $\beta \in -\De{\lambda}$; 
\item $\alpha \in -\Dn{\lambda}$ and $\beta \in -\De{\lambda}$. 
\end{enu}
In cases (1)--(3) and (6), it is obvious that $\alpha \prec \alpha + \beta \prec \beta$ 
since $\vtl$ is a reflection order on $\Delta^{+}$. 

Assume case (4) or (5). Suppose, for a contradiction, that $\alpha \prec \alpha + \beta \prec \beta$ does not hold. 
If $\alpha \prec \beta \prec \alpha + \beta$, then $\alpha + \beta \in (-\Dn{\lambda}) \sqcup (-\De{\lambda})$. 
By the definition of ${\prec} = {\prec_{\vtl}}$, we deduce that $-\beta \vtl -\alpha - \beta \vtl \alpha$. 
However, since $-\beta = (-\alpha - \beta) + \alpha$, 
this contradicts the assumption that $\vtl$ is a reflection order on $\Delta^{+}$. 
Similarly, we can show that $\alpha + \beta \prec \alpha \prec \beta$ does not hold. 
Hence we have $\alpha \prec \alpha + \beta \prec \beta$, as desired. 
This completes the proof that ${\prec} \in \RO(\lambda, w(\lambda) \Delta^{+})$. 

Now, we claim that the assignment ${\vtl} \mapsto {\prec_{\vtl}}$ gives a bijection from $\RO(\lambda, \Delta^{+})$ to $\RO(\lambda, w(\lambda) \Delta^{+})$. 
Indeed, the construction of the inverse map ${\prec_{\vtl}} \mapsto {\vtl}$ is obvious, and it is easy to verify that the resulting total order $\vtl$ is a reflection order on $\Delta^{+}$ by the same argument as above. This proves the proposition. 
\end{proof}

\begin{remark} \label{rem:nonempty_RO}
Through the bijection $\RO(\lambda, \Delta^{+}) \xrightarrow{\sim} \RO(\lambda, w(\lambda)\Delta^{+}) (\not= \emptyset)$ given by Proposition~\ref{prop:bij_RO}, we deduce that $\RO(\lambda, \Delta^{+})$ is also nonempty. 
\end{remark}

%%%%%%%%%%%%%%%%%%%%%%%%%%
% sec:quantum_alcove_model %
%%%%%%%%%%%%%%%%%%%%%%%%%%

\section{The quantum alcove model}\label{sec:quantum_alcove_model}
We review the quantum alcove model, introduced in \cite{LL1} and \cite{LNS}. 

\subsection{Alcove paths and admissible subsets}
First, we recall from \cite{LP1} the definition of alcove paths. 
For $\alpha \in \Delta$ and $k \in \BZ$, we set $H_{\alpha, k} := \{\nu \in \Fh_{\BR}^{\ast} \mid \pair{\nu}{\alpha^{\vee}} = k\}$; 
$H_{\alpha, k}$ is a hyperplane in $\Fh_{\BR}^{\ast}$. 
Also, for $\alpha \in \Delta$ and $k \in \BZ$, we denote by $s_{\alpha, k}$ the (affine) reflection with respect to $H_{\alpha, k}$. 
Note that $s_{\alpha, k}(\nu) = \nu - (\pair{\nu}{\alpha^{\vee}} - k)\alpha$ for $\nu \in \Fh_{\BR}^{\ast}$. 
Each connected component of the space 
\begin{equation}
\Fh_{\BR}^{\ast} \setminus \bigcup_{\alpha \in \Delta^{+}, k \in \BZ} H_{\alpha, k}
\end{equation}
is called an \emph{alcove}. 
If two alcoves $A$ and $B$ have a common wall, then we say that $A$ and $B$ are \emph{adjacent}. 
For adjacent alcoves $A$ and $B$, we write $A \xrightarrow{\beta} B$, $\beta \in \Delta$, 
if the common wall of $A$ and $B$ is contained in $H_{\beta, k}$ for some $k \in \BZ$, 
and $\beta$ points in the direction from $A$ to $B$. 

\begin{definition}[{\cite[Definition~5.2]{LP1}}]
A sequence $(A_{0}, \ldots, A_{r})$ of alcoves is called an \emph{alcove path} if $A_{i-1}$ and $A_{i}$ are adjacent for all $i = 1, \ldots, r$. 
If the length $r$ of an alcove path $(A_{0}, \ldots, A_{r})$ is minimal among all alcove paths from $A_{0}$ to $A_{r}$, 
we say that $(A_{0}, \ldots, A_{r})$ is \emph{reduced}. 
\end{definition}

The \emph{fundamental alcove} $A_{\circ}$ is defined by 
\begin{equation}
A_{\circ} := \{ \nu \in \mathfrak{h}_{\mathbb{R}}^{\ast} \mid \text{$0 < \langle \nu, \alpha^{\vee} \rangle < 1$ for all  $\alpha \in \Delta^{+}$} \}. 
\end{equation}
Also, for $\lambda \in P$, we define $A_{\lambda}$ by 
\begin{equation}
A_{\lambda} := A_{\circ} + \lambda = \{ \nu + \lambda \mid \nu \in A_{\circ} \}. 
\end{equation}

\begin{definition}[{\cite[Definition~5.4]{LP1}}]
Let $\lambda \in P$. 
A sequence $(\beta_{1}, \ldots, \beta_{r})$ of roots $\beta_{1}, \ldots, \beta_{r} \in \Delta$ 
is called a \emph{$\lambda$-chain} 
if there exists an alcove path $(A_{0}, \ldots, A_{r})$, with $A_{0} = A_{\circ}$ and $A_{r} = A_{-\lambda}$, such that 
\begin{equation}
A_{\circ} = A_{0} \xrightarrow{-\beta_{1}} A_{1} \xrightarrow{-\beta_{2}} \cdots \xrightarrow{-\beta_{r}} A_{r} = A_{-\lambda}. 
\end{equation}
If such an alcove path $(A_{0}, \ldots, A_{r})$ is reduced, then we call $(\beta_{1}, \ldots, \beta_{r})$ a \emph{reduced} $\lambda$-chain. 
\end{definition}

Now, following \cite[Section~3.2]{LNS}, we review the quantum alcove model. 
\begin{definition}[{\cite[Definition~17]{LNS}}]
Let $\lambda \in P$, and let $\Gamma = (\beta_{1}, \ldots, \beta_{r})$ be a $\lambda$-chain. 
Fix $w \in W$. A subset $A = \{j_{1} < \cdots < j_{p}\} \subset \{1, \ldots, r\}$ is said to be \emph{$w$-admissible} if 
\begin{equation}
\bp(A): w = w_{0} \xrightarrow{|\beta_{j_{1}}|} w_{1} \xrightarrow{|\beta_{j_{2}}|} \cdots \xrightarrow{|\beta_{j_{p}}|} w_{p}
\end{equation}
is a directed path in $\QBG(W)$. 
Let $\CA(w, \Gamma)$ denote the set of all $w$-admissible subsets of $\{1, \ldots, r\}$. 
\end{definition}
\begin{remark}
The original definition of admissible subsets in \cite{LL1} is only for $w = e \in W$. 
The notion of $w$-admissible subsets for an arbitrary $w \in W$ is introduced in \cite{LNS}. 
\end{remark}

We also consider the ``$q = 0$ part'' of $\CA(w, \Gamma)$. 
For $\lambda \in P$, let $\Gamma = (\beta_{1}, \ldots, \beta_{r})$ be a $\lambda$-chain. Fix $w \in W$. 
We set 
\begin{equation}
\CA|_{q = 0}(w, \Gamma) := \{ A \in \CA(w, \Gamma) \mid \text{$\bp(A)$ is a directed path in $\BG(W)$} \}. 
\end{equation}

For $\lambda \in P$, let $\Gamma = (\beta_{1}, \ldots, \beta_{r})$ be a $\lambda$-chain.
By the definition of $\lambda$-chains, there exists an alcove path $(A_{\circ} = A_{0}, \ldots, A_{r} = A_{-\lambda})$ such that 
\begin{equation}
A_{\circ} = A_{0} \xrightarrow{-\beta_{1}} A_{1} \xrightarrow{-\beta_{2}} \cdots \xrightarrow{-\beta_{r}} A_{r} = A_{-\lambda}. 
\end{equation}
For $k = 1, \ldots, r$, we take $l_{k} \in \BZ$ such that $H_{\beta_{k}, -l_{k}}$ contains the common wall of $A_{k-1}$ and $A_{k}$, 
and set $\wti{l_{k}} := \pair{\lambda}{\beta_{k}^{\vee}} - l_{k}$. 

Fix $w \in W$. For $A = \{j_{1} < \cdots < j_{p}\} \in \CA(w, \Gamma)$, we set 
\begin{equation}
\ed(A) := ws_{|\beta_{j_{1}}|} \cdots s_{|\beta_{j_{p}}|}, \quad \wt(A) := -ws_{\beta_{j_{1}}, -l_{j_{1}}} \cdots s_{\beta_{j_{p}}, -l_{j_{p}}} (-\lambda); 
\end{equation} 
we call $\wt(A)$ the \emph{weight} of $A$.
Also, we define a subset $A^{-} \subset A$ by 
\begin{equation}
A^{-} := \left\{ j_{k} \in A \ \middle| \ \text{$ws_{|\beta_{j_{1}}|} \cdots s_{|\beta_{j_{k-1}}|} \xrightarrow{|\beta_{j_{k}}|} ws_{|\beta_{j_{1}}|} \cdots s_{|\beta_{j_{k}}|}$ is a quantum edge} \right\},  
\end{equation}
and set 
\begin{equation}
\down(A) := \sum_{j \in A^{-}} |\beta_{j}|^{\vee}, \quad \height(A) := \sum_{j \in A^{-}} \sgn(\beta_{j})\wti{l_{j}}; 
\end{equation}
note that $\ed(A) = \ed(\bp(A))$ and $\down(A) = \wt(\bp(A))$. 
In addition, we define $n(A) \in \BZ_{\ge 0}$ by $n(A) := \# \{ j \in A \mid \beta_{j} \in -\Delta^{+} \}$.

\begin{remark}
If $A \in \CA|_{q=0}(w, \Gamma)$, then we have $A^{-} = \emptyset$, and hence $\down(A) = 0$, $\height(A) = 0$. 
\end{remark}

\begin{example}
Assume that $\Fg$ is of type $A_{2}$. Let $\Gamma = (\alpha_{2}, \alpha_{1}+\alpha_{2}, \alpha_{2}, -\alpha_{1})$. Then $\Gamma$ is a $(-\vpi_{1}+2\vpi_{2})$-chain. 
In Table~\ref{tab:admissible}, we give the list of all $A \in \CA(s_{1}, \Gamma)$, together with $\ed(A)$, $\down(A)$, $\wt(A)$, and $\height(A)$. 
\begin{table}[ht]
\centering
\caption{The list of $\ed(A)$, $\down(A)$, $\wt(A)$, and $\height(A)$ for all $A \in \CA(s_{1}, \Gamma)$}
\label{tab:admissible}
\begin{tabular}{|c||cccc|} \hline
$A$ & $\ed(A)$ & $\down(A)$ & $\wt(A)$ & $\height(A)$ \\ \hline 
$\emptyset$ & $s_{1}$ & $0$ & $\vpi_{1}+\vpi_{2}$ & $0$ \\ 
$\{1\}$ & $s_{1}s_{2}$ & $0$ & $-\vpi_{1}-\vpi_{2}$ & $0$ \\ 
$\{2\}$ & $s_{2}s_{1}$ & $0$ & $2\vpi_{1}-\vpi_{2}$ & $0$ \\ 
$\{3\}$ & $s_{1}s_{2}$ & $0$ & $0$ & $0$ \\
$\{4\}$ & $e$ & $\alpha_{1}^{\vee}$ & $\vpi_{1}+\vpi_{2}$ & $0$ \\ 
$\{1, 3\}$ & $s_{1}$ & $\alpha_{2}^{\vee}$ & $0$ & $1$ \\ 
$\{1, 4\}$ & $s_{1}s_{2}s_{1}$ & $0$ & $-\vpi_{1}-\vpi_{2}$ & $0$ \\ 
$\{2, 3\}$ & $s_{1}s_{2}s_{1}$ & $0$ & $0$ & $0$ \\ 
$\{2, 4\}$ & $s_{2}$ & $\alpha_{1}^{\vee}$ & $2\vpi_{1}-\vpi_{2}$ & $0$ \\ 
$\{3, 4\}$ & $s_{1}s_{2}s_{1}$ & $0$ & $0$ & $0$ \\ 
$\{1, 3, 4\}$ & $e$ & $\alpha_{1}^{\vee}+\alpha_{2}^{\vee}$ & $0$ & $1$ \\ 
$\{2, 3, 4\}$ & $s_{1}s_{2}$ & $\alpha_{1}^{\vee}$ & $0$ & $0$ \\ \hline
\end{tabular}
\end{table}
\end{example}

\subsection{Reduced chains of roots}\label{subsec:reduced_expression_vs_chain_of_roots}
Let $\lambda \in P$. We describe the relation between reduced $\lambda$-chains 
and reduced expressions for $t_{\lambda} \in \Wex^{\vee}$.

Take a reduced expression $t_{-\lambda} = \pi^{\vee} s_{i_1} s_{i_2} \cdots s_{i_r} \in \Wex^{\vee}$, 
where $i_1, \ldots, i_r \in I_{\af}$ and $\pi^{\vee} \in \Omega^{\vee}$ is an element of length zero. For $k \in \{1, \ldots, r \}$, (borrowing notation from \cite[Section~3.1]{NNS}) we set 
\begin{align}
(\beta_{k}^{\sL})^{\vee} := \pi^{\vee} s_{i_1} \cdots s_{i_{k-1}} \alpha_{i_k}^\vee,
\end{align}
and $\gamma_{k}^{\sL} := \overline{\beta_{k}^{\sL}} (=\overline{((\beta_{k}^{\sL})^{\vee})^{\vee}})$. 
Note that $\{ (\beta_{1}^{\sL})^{\vee}, \ldots, (\beta_{r}^{\sL})^{\vee} \} = \Daf{\lambda}$.
Then, \cite[Lemma~5.3]{LP1} implies the following.
\begin{lemma}
The sequence $(\gamma_{1}^{\sL}, \ldots, \gamma_{r}^{\sL})$ is a reduced $\lambda$-chain.
\end{lemma}

In addition, for the alcove path $(A_{0}, \ldots, A_{r})$, with $A_{0} = A_{\circ}$ and $A_{r} = A_{-\lambda}$, corresponding to the $\lambda$-chain $(\gamma_{1}^{\sL}, \ldots, \gamma_{r}^{\sL})$, i.e., for the alcove path 
\begin{align}
A_{\circ} = A_0 \xrightarrow{-\gamma_{1}^{\sL}} A_{1} \xrightarrow{-\gamma_{2}^{\sL}} \cdots \xrightarrow{-\gamma_{r}^{\sL}} A_{r} = A_{-\lambda}, 
\end{align}
it is easy to verify that the hyperplane containing the common wall of $A_{k-1}$ and $A_{k}$ 
is identical to $H_{\gamma_{k}^{\sL}, -\deg((\beta_{k}^{\sL})^{\vee})}$ 
for each $k = 1, \ldots, r$. 

Conversely, take a reduced $\lambda$-chain $(\beta_{1}, \ldots, \beta_{r})$. 
Then there exists a reduced alcove path $(A_{0}, \ldots, A_{r})$, with $A_{0} = A_{\circ}$ and $A_{r} = A_{-\lambda}$, such that 
\begin{align}
A_{\circ} = A_0 \xrightarrow{-\beta_{1}} A_{1} \xrightarrow{-\beta_{2}} \cdots \xrightarrow{-\beta_{r}} A_{r} = A_{-\lambda}. 
\end{align}
By \cite[Lemma~5.3]{LP1}, there exists $v \in \Waf^{\vee}$ with reduced expression 
$v = s_{i_{1}} \cdots s_{i_{r}}$ such that $v(A_{\circ}) = A_{-\lambda}$, 
and such that $A_{k} = s_{i_{1}} \cdots s_{i_{k}}(A_{\circ})$ for $k = 0, \ldots, r$. Since $t_{-\lambda}(A_{\circ}) = A_{-\lambda}$, 
there exists an element $\pi^{\vee} \in \Omega^{\vee}$ of length zero such that $t_{-\lambda} = \pi^{\vee} v$. 
Hence we obtain a reduced expression $t_{-\lambda} = \pi^{\vee} s_{i_{1}} \cdots s_{i_{r}}$ for $t_{-\lambda}$. 
Thus, the set of reduced $\lambda$-chains and 
the set of reduced expressions for $t_{\lambda} \in \Wex^{\vee}$ are in one-to-one correspondence. 

Next, we construct a ``suitable'' reduced expression $t_{\lambda} = s_{i_{r}} \cdots s_{i_{2}}s_{i_{1}} (\pi^{\vee})^{-1}$ 
and a ``suitable'' $\lambda$-chain $\Gamma_{\vtl}(\lambda)$; 
see the paragraph at the end of this subsection.
For this purpose, we need the following lemmas. 

%%%%%%%%%%%%%%%%%%%%%%%%%%
% lem:map_for_forgetful_map1 %
%%%%%%%%%%%%%%%%%%%%%%%%%%

\begin{lemma}\label{lem:map_for_forgetful_map1}
For $\beta^{\vee} \in \Daf{\lambda}$, we have $\deg(\beta^{\vee})/\pair{\lambda}{\overline{\beta^{\vee}}} \in \BQ \cap [0, 1]$.
\end{lemma}
\begin{proof}
Let $\chi: \Delta^{\vee} = \{ \gamma^{\vee} \mid \gamma \in \Delta \} \rightarrow \{0, 1\}$ denote the characteristic function of $-\Delta^{\vee, +}$; i.e., 
\begin{equation}
\chi(\gamma^{\vee}) := \begin{cases} 0 & \text{if $\gamma^{\vee} \in \Delta^{\vee, +}$}, \\ 1 & \text{if $\gamma^{\vee} \in -\Delta^{\vee, +}$}. \end{cases} 
\end{equation}
By \cite[equation~(1) in the proof of (2.4.1)]{Mac}, 
we have ($\pair{\lambda}{\overline{\beta^{\vee}}} \not= 0$ and) 
\begin{align}
\Daf{\lambda} = \{ \alpha^{\vee} + k \wti{\delta} \mid k \in \BZ, \ \chi(\alpha^{\vee}) \leq k < \chi(\alpha^{\vee}) + \pair{\lambda}{\alpha^{\vee}} \}. \label{eq:Inv}
\end{align}
Since $\chi(\overline{\beta^{\vee}}) = 0$ or $1$, 
we see that $0 \leq \deg(\beta^{\vee}) \leq \pair{\lambda}{\overline{\beta^{\vee}}}$, 
and hence $0 \leq \deg(\beta^{\vee})/\pair{\lambda}{\overline{\beta^{\vee}}} \leq 1$. 
Also, it is obvious that $\deg(\beta^{\vee})/\pair{\lambda}{\overline{\beta^{\vee}}} \in \BQ$ 
since $\deg(\beta^{\vee}) \in \BZ$ and $\pair{\lambda}{\overline{\beta^{\vee}}} \in \BZ$. This proves the lemma. 
\end{proof}

%%%%%%%%%%%%%%%%%%%%%%%%%%
% lem:map_for_forgetful_map2 %
%%%%%%%%%%%%%%%%%%%%%%%%%%

\begin{lemma}\label{lem:map_for_forgetful_map2}
Let $\beta^\vee \in \Daf{\lambda}$. 
Then we have $\overline{\beta} \in (\Delta^{+}(\lambda)_{>0}) \sqcup (- \Delta^{+}(\lambda)_{<0})$. 
\end{lemma}
\begin{proof}
By \eqref{eq:Inv}, we see that $\pair{\lambda}{\overline{\beta^{\vee}}} > 0$. 
If $\overline{\beta} \in \Delta^{+}$, then we have $\overline{\beta} \in \Delta^{+}(\lambda)_{>0}$. 
If $\overline{\beta} \in -\Delta^{+}$, then we have $-\overline{\beta} \in \Delta^{+}(\lambda)_{<0}$. 
This proves the lemma. 
\end{proof}

By Lemmas~\ref{lem:map_for_forgetful_map1} and \ref{lem:map_for_forgetful_map2}, we can define a map $\Phi : \Daf{\lambda} \rightarrow \BQ_{\ge 0} \times ((\Delta^{+}(\lambda)_{>0}) \sqcup (-\Delta^{+}(\lambda)_{<0}))$ by
\begin{align}
\Phi(\beta^{\vee}) := \left( \frac{\deg(\beta^{\vee})}{\pair{\lambda}{\overline{\beta^{\vee}}}}, \overline{\beta} \right); 
\end{align}
it is obvious that $\Phi$ is injective. 

Take a reflection order ${\vtl} \in \RO(\lambda, \Delta^{+})$, and let 
${\prec} \in \RO(\lambda, w(\lambda)\Delta^{+})$ be the reflection order on $w(\lambda)\Delta^{+}$ corresponding to $\vtl$ 
by the bijection $\RO(\lambda, \Delta^{+}) \rightarrow \RO(\lambda, w(\lambda)\Delta^{+})$, given by Proposition~\ref{prop:bij_RO}. 
We define a total order $\prec^{*}$ on $w(\lambda)\Delta^{+}$ as follows: $\alpha \prec^{*} \beta$ if and only if $\beta \prec \alpha$. 
Then, $\prec^{*}$ is also a reflection order on $w(\lambda) \Delta^{+}$, under which we regard $(\Dp{\lambda}) \sqcup (-\Dn{\lambda}) \subset w(\lambda)\Delta^{+}$ as a totally ordered set. 
By using the usual total order on $\BQ_{\ge 0}$, we can define the lexicographic order on $\BQ_{\ge 0} \times ((\Dp{\lambda}) \sqcup (-\Dn{\lambda}))$, 
by which $\BQ_{\ge 0} \times ((\Dp{\lambda}) \sqcup (-\Dn{\lambda}))$ is a totally ordered set. 
Thus, the injection $\Phi$ induces a total order on $\Daf{\lambda}$, denoted by $<$. 
We see that $<$ is an affine reflection order on $\Daf{\lambda}$ in the sense below. 

\begin{definition}[{cf. \cite[p.~662, Theorem]{Papi}}] \label{def:reflection_order}
Let $L \subset \Delta_{\af}^{\vee, +}$ be a totally ordered set, and $<$ a total order on $L$. 
An order $<$ is called an \textit{affine reflection order} if the following hold: 
\begin{enu}
\item If $\alpha^{\vee} + \beta^{\vee} \in \Delta_{\af}^{\vee, +}$ for $\alpha^{\vee}, \beta^{\vee} \in L$, then $\alpha^{\vee} + \beta^{\vee} \in L$. In addition, either $\alpha^{\vee} < \alpha^{\vee} + \beta^{\vee} < \beta^{\vee}$ or $\beta^{\vee} < \alpha^{\vee} + \beta^{\vee} < \alpha^{\vee}$ holds; 
\item If $\alpha^{\vee} + \beta^{\vee} \in L$ for $\alpha^{\vee}, \beta^{\vee} \in \Delta_{\af}^{\vee}$, then either of the following holds: (a) $\alpha^{\vee} \in L$ and $\alpha^{\vee} < \alpha^{\vee} + \beta^{\vee}$, or (b) $\beta^{\vee} \in L$ and $\beta^{\vee} < \alpha^{\vee} + \beta^{\vee}$. 
\end{enu}
\end{definition}

\begin{proposition}[{cf. \cite[Proposition~3.1.8]{NNS}}]
The total order $<$ on $\Daf{\lambda}$ induced by the injection $\Phi$ is an affine reflection order. 
\end{proposition}
\begin{proof}
First, we show condition (1) in Definition~\ref{def:reflection_order}. 
Let $\alpha^{\vee}, \beta^{\vee} \in \Daf{\lambda}$, 
and assume that $\alpha^{\vee} + \beta^{\vee} \in \Delta_{\af}^{\vee, +}$. Then it is obvious that $\alpha^{\vee} + \beta^{\vee} \in \Daf{\lambda}$. 
Assume that $\alpha^{\vee} < \beta^{\vee}$. 
By the definition of $<$ on $\Daf{\lambda}$, we have either of the following: 
\begin{enu}
\item $\deg(\alpha^{\vee})/\pair{\lambda}{\overline{\alpha^{\vee}}} < \deg(\beta^{\vee})/\pair{\lambda}{\overline{\alpha^{\vee}}}$, 
\item $\deg(\alpha^{\vee})/\pair{\lambda}{\overline{\alpha^{\vee}}} = \deg(\beta^{\vee})/\pair{\lambda}{\overline{\beta^{\vee}}}$ and $\overline{\beta} \prec \overline{\alpha}$ (i.e., $\overline{\alpha} \prec^{*} \overline{\beta}$). 
\end{enu}
If (1) holds, then we have 
\begin{equation}
\frac{\deg(\alpha^{\vee})}{\pair{\lambda}{\overline{\alpha^{\vee}}}} < \frac{\deg(\alpha^{\vee}) + \deg(\beta^{\vee})}{\pair{\lambda}{\overline{\alpha^{\vee}}} + \pair{\lambda}{\overline{\beta^{\vee}}}} = \frac{\deg(\alpha^{\vee}+\beta^{\vee})}{\pair{\lambda}{\overline{\alpha^{\vee}+\beta^{\vee}}}} < \frac{\deg(\beta^{\vee})}{\pair{\lambda}{\overline{\beta^{\vee}}}}, 
\end{equation}
which implies that $\alpha^{\vee} < \alpha^{\vee}+\beta^{\vee} < \beta^{\vee}$. 
If (2) holds, then we have 
\begin{equation}
\frac{\deg(\alpha^{\vee})}{\pair{\lambda}{\overline{\alpha^{\vee}}}} = \frac{\deg(\alpha^{\vee}) + \deg(\beta^{\vee})}{\pair{\lambda}{\overline{\alpha^{\vee}}} + \pair{\lambda}{\overline{\beta^{\vee}}}} = \frac{\deg(\alpha^{\vee}+\beta^{\vee})}{\pair{\lambda}{\overline{\alpha^{\vee}+\beta^{\vee}}}} = \frac{\deg(\beta^{\vee})}{\pair{\lambda}{\overline{\beta^{\vee}}}}, 
\end{equation}
and $\overline{\beta} \prec (\overline{\alpha^{\vee}+\beta^{\vee}})^{\vee} \prec \overline{\alpha}$ by Corollary~\ref{cor:reflection_coroot_arbitrary}, since $\prec$ is a reflection order. 
Hence we deduce that $\alpha^{\vee} < \alpha^{\vee}+\beta^{\vee} < \beta^{\vee}$, as desired. 

Next, we show condition (2) in Definition~\ref{def:reflection_order}. 
Assume that $\alpha^{\vee} + \beta^{\vee} \in \Daf{\lambda}$. 
If $\alpha^{\vee}, \beta^{\vee} \in \Daf{\lambda}$, then the assertion is obvious by condition (1) in Definition~\ref{def:reflection_order}. 
If $\alpha^{\vee}, \beta^{\vee} \in \Delta_{\af}^{\vee, +} \setminus t_{\lambda}^{-1}\Delta_{\af}^{\vee, -}$ and $\alpha^{\vee} + \beta^{\vee} \in \Delta_{\af}^{\vee}$, 
then it is obvious that $\alpha^{\vee} + \beta^{\vee} \in \Delta_{\af}^{\vee, +} \setminus t_{\lambda}^{-1}\Delta_{\af}^{\vee, -}$, which is 
a contradiction. 
Hence we may assume that $\alpha^{\vee} \in \Daf{\lambda}$ and $\beta^{\vee} \notin \Daf{\lambda}$. 
We will show that $\alpha^{\vee} < \alpha^{\vee} + \beta^{\vee}$. 

First, consider the case that $\overline{\beta} \in \Delta^{+}$. 
Recall \eqref{eq:Inv}. It follws that $0 \le \deg(\alpha^{\vee})/\pair{\lambda}{\overline{\alpha^{\vee}}} \le 1$ 
by the assumption that $\alpha^{\vee} \in \Daf{\lambda}$. 
Also, by the assumption that $\beta^{\vee} \notin \Daf{\lambda}$, we have the following possibilities: 
(a): $\pair{\lambda}{\overline{\beta^{\vee}}} < 0 \le \deg(\beta^{\vee})$; 
(b-1): $0 \le \pair{\lambda}{\overline{\beta^{\vee}}} \le \deg(\beta^{\vee})$ and $\deg(\beta^{\vee}) \not= 0$; 
(b-2): $\pair{\lambda}{\overline{\beta^{\vee}}} = \deg(\beta^{\vee}) = 0$. 

Assume that (a) holds. Then we see that 
\begin{equation}
\frac{\deg(\alpha^{\vee} + \beta^{\vee})}{\pair{\lambda}{\overline{\alpha^{\vee} + \beta^{\vee}}}} = \frac{\deg(\alpha^{\vee}) + \deg(\beta^{\vee})}{\pair{\lambda}{\overline{\alpha^{\vee}}} + \pair{\lambda}{\overline{\beta^{\vee}}}} > \frac{\deg(\alpha^{\vee}) + \deg(\beta^{\vee})}{\pair{\lambda}{\overline{\alpha^{\vee}}}} \ge \frac{\deg(\alpha^{\vee})}{\pair{\lambda}{\overline{\alpha^{\vee}}}}. 
\end{equation}
Hence, by the definition of $<$ on $\Daf{\lambda}$, it follows that $\alpha^{\vee} < \alpha^{\vee} + \beta^{\vee}$. 

Next, assume that (b-1) holds. If $\pair{\lambda}{\overline{\beta^{\vee}}} = 0$, then the assertion is obvious since 
\begin{equation}
\frac{\deg(\alpha^{\vee} + \beta^{\vee})}{\pair{\lambda}{\overline{\alpha^{\vee} + \beta^{\vee}}}} = \frac{\deg(\alpha^{\vee}) + \deg(\beta^{\vee})}{\pair{\lambda}{\overline{\alpha^{\vee}}}} > \frac{\deg(\alpha^{\vee})}{\pair{\lambda}{\overline{\alpha^{\vee}}}}. 
\end{equation}
Assume that $\pair{\lambda}{\overline{\beta^{\vee}}} > 0$. Since 
\begin{equation}
\frac{\deg(\alpha^{\vee})}{\pair{\lambda}{\overline{\alpha^{\vee}}}} \le 1 \le \frac{\deg(\beta^{\vee})}{\pair{\lambda}{\overline{\beta^{\vee}}}}, 
\end{equation}
we deduce that 
\begin{equation}
\frac{\deg(\alpha^{\vee}) + \deg(\beta^{\vee})}{\pair{\lambda}{\overline{\alpha^{\vee}}} + \pair{\lambda}{\overline{\beta^{\vee}}}} \ge \frac{\deg(\alpha^{\vee})}{\pair{\lambda}{\overline{\alpha^{\vee}}}}, 
\end{equation}
where the equality holds if and only if 
\begin{equation}
\frac{\deg(\alpha^{\vee})}{\pair{\lambda}{\overline{\alpha^{\vee}}}} = 1 = \frac{\deg(\beta^{\vee})}{\pair{\lambda}{\overline{\beta^{\vee}}}}. \label{eq:equal_condition} %eq:equal_condition
\end{equation}
If \eqref{eq:equal_condition} fails to hold, then the assertion follows by the definition of $<$ on $\Daf{\lambda}$. 
Let us assume that \eqref{eq:equal_condition} holds. 
We have 
\begin{equation}
\deg(\alpha^{\vee}) = \pair{\lambda}{\overline{\alpha^{\vee}}}, \label{eq:equal_1}
\end{equation}
\begin{equation}
\deg(\alpha^{\vee} + \beta^{\vee}) = \pair{\lambda}{\overline{\alpha^{\vee} + \beta^{\vee}}}, \label{eq:equal_2}
\end{equation}
\begin{equation}
\deg(\beta^{\vee}) = \pair{\lambda}{\overline{\beta^{\vee}}} > 0. \label{eq:equal_3}
\end{equation}
Suppose, for a contradiction, that $\overline{\alpha} \prec (\overline{\alpha^{\vee}+\beta^{\vee}})^{\vee}$. 
Then, \eqref{eq:equal_1} and \eqref{eq:Inv} imply that $\overline{\alpha} \in -\Dn{\lambda}$ since $\alpha^{\vee} \in \Daf{\lambda}$. 
Similarly, we have $(\overline{\alpha^{\vee}+\beta^{\vee}})^{\vee} \in -\Dn{\lambda}$ by \eqref{eq:equal_2} and \eqref{eq:Inv}. 
Here, recall that $\overline{\beta} \in \Delta^{+}$. 
Hence we see that $\overline{\beta} \in \Dp{\lambda}$. 
By \eqref{eq:def_RO_2}, we deduce that $\overline{\beta} \prec \overline{\alpha} \prec (\overline{\alpha^{\vee}+\beta^{\vee}})^{\vee}$. 
This contradicts that $\prec$ is a reflection order by Corollary~\ref{cor:reflection_coroot_arbitrary}. 
Thus we have $(\overline{\alpha^{\vee}+\beta^{\vee}})^{\vee} \prec \overline{\alpha}$. Therefore, we conclude that $\overline{\alpha} \prec^{*} (\overline{\alpha^{\vee}+\beta^{\vee}})^{\vee}$ 
and $\alpha^{\vee} < \alpha^{\vee} + \beta^{\vee}$, as desired. 

Now, assume that (b-2) holds. Then we have 
\begin{equation}
\frac{\deg(\alpha^{\vee} + \beta^{\vee})}{\pair{\lambda}{\overline{\alpha^{\vee}+\beta^{\vee}}}} = \frac{\deg(\alpha^{\vee})}{\pair{\lambda}{\overline{\alpha^{\vee}}}}. 
\end{equation}
Hence we need to show that $(\overline{\alpha^{\vee}+\beta^{\vee}})^{\vee} \prec \overline{\alpha}$. 
Since $\pair{\lambda}{\overline{\beta^{\vee}}} = 0$ and $\overline{\beta} \in \Delta^{+}$, 
we have $\overline{\beta} \in \De{\lambda}$. 
Also, since $\alpha^{\vee}+\beta^{\vee} \in \Daf{\lambda}$, we deduce from Lemma~\ref{lem:map_for_forgetful_map2} that $(\overline{\alpha^{\vee}+\beta^{\vee}})^{\vee} \in (\Dp{\lambda}) \sqcup (-\Dn{\lambda})$. 
By \eqref{eq:def_RO_2}, we see that $(\overline{\alpha^{\vee}+\beta^{\vee}})^{\vee} \prec -\overline{\beta}$. 
Since $\prec$ is a reflection order, Lemma~\ref{lem:reflection_coroot} implies that $(\overline{\alpha^{\vee}+\beta^{\vee}})^{\vee} \prec \overline{\alpha} \prec -\overline{\beta}$, as desired. 

Finally, we consider the case that $\overline{\beta} \in \Delta^{-}$. 
In this case, \eqref{eq:Inv} implies that $\deg(\beta^{\vee}) > \pair{\lambda}{\overline{\beta^{\vee}}}$. 
Note that $\deg(\beta^{\vee}) > 0$ since $\beta^{\vee} \in \Delta_{\af}^{\vee, +}$. 
Hence the assertion follows by the same argument as in the case that (a) holds, or (b-1) holds, with $\deg(\beta^{\vee}) \not= \pair{\lambda}{\overline{\beta^{\vee}}}$. 
This proves the proposition. 
\end{proof}

Thanks to this proposition, we can apply \cite[p.~662, Theorem]{Papi} to deduce that there exists a (unique) reduced expression $t_{\lambda} = s_{i_r} \cdots s_{i_2} s_{i_1} (\pi^\vee)^{-1}$ 
such that $\beta_{1}^{\sL} < \cdots < \beta_{r}^{\sL}$; 
such a reduced expression is said to be \textit{suitable} for $\vtl$. 
Also, we can take the $\lambda$-chain $\Gamma_{\vtl}(\lambda) = (\gamma_{1}^{\sL}, \ldots, \gamma_{r}^{\sL})$, which corresponds to the suitable reduced expression $t_{\lambda} = s_{i_r} \cdots s_{i_2} s_{i_1} (\pi^\vee)^{-1}$ 
by the relation between reduced expressions and reduced $\lambda$-chains, described in Section~\ref{subsec:reduced_expression_vs_chain_of_roots}; 
we also say that $\Gamma_{\vtl}(\lambda)$ is \emph{suitable} for $\vtl$. 

\section{Interpolated QLS paths}\label{sec:def_IQLS}

We introduce interpolated QLS paths, which generalize quantum Lakshmibai-Seshadri paths (QLS paths for short), 
defined in \cite[Section~3]{LNSSS2}. 
We fix $\lambda \in P$ and ${\vtl} \in \RO(\lambda, \Delta^{+})$. 

\subsection{Integrality conditions} 

In order to introduce interpolated QLS paths, we need the following.

\begin{definition} \label{def:integrality}
Let $x, y \in W$, and $\sigma \in \BQ$. 
\begin{enu}
\item If there exists a directed path $x = x_{0} \xrightarrow{\gamma_{1}} x_{1} \xrightarrow{\gamma_{2}} \cdots \xrightarrow{\gamma_{r}} x_{r} = y$ in $\QBG(W)$
such that $\gamma_{1}, \ldots, \gamma_{r} \in \Dp{\lambda}$, $\gamma_{1} \vtr \cdots \vtr \gamma_{r}$, 
and $\sigma \pair{\lambda}{\gamma_{k}^{\vee}} \in \BZ$ for all $k = 1, \ldots, r$, 
then we write $x \xRightarrow{(\lambda, +)}_{\raisebox{1ex}{\scriptsize{$\sigma$}}} y$. 
\item If there exists a directed path $x = x_{0} \xrightarrow{\gamma_{1}} x_{1} \xrightarrow{\gamma_{2}} \cdots \xrightarrow{\gamma_{r}} x_{r} = y$ in $\BG(W)$
such that $\gamma_{1}, \ldots, \gamma_{r} \in \Dp{\lambda}$, $\gamma_{1} \vtr \cdots \vtr \gamma_{r}$, 
and $\sigma \pair{\lambda}{\gamma_{k}^{\vee}} \in \BZ$ for all $k = 1, \ldots, r$, 
then we write $x \xRightarrow{(\lambda, +, q=0)}_{\raisebox{1ex}{\scriptsize{$\sigma$}}} y$. 
\item If there exists a directed path $x = x_{0} \xrightarrow{\gamma_{1}} x_{1} \xrightarrow{\gamma_{2}} \cdots \xrightarrow{\gamma_{r}} x_{r} = y$ in $\QBG(W)$
such that $\gamma_{1}, \ldots, \gamma_{r} \in \Dn{\lambda}$, $\gamma_{1} \vtr \cdots \vtr \gamma_{r}$, 
and $\sigma \pair{\lambda}{\gamma_{k}^{\vee}} \in \BZ$ for all $k = 1, \ldots, r$, 
then we write $x \xRightarrow{(\lambda, -)}_{\raisebox{1ex}{\scriptsize{$\sigma$}}} y$. 
\item If there exists a directed path $x = x_{0} \xrightarrow{\gamma_{1}} x_{1} \xrightarrow{\gamma_{2}} \cdots \xrightarrow{\gamma_{r}} x_{r} = y$ in $\BG(W)$
such that $\gamma_{1}, \ldots, \gamma_{r} \in \Dn{\lambda}$, $\gamma_{1} \vtr \cdots \vtr \gamma_{r}$, 
and $\sigma \pair{\lambda}{\gamma_{k}^{\vee}} \in \BZ$ for all $k = 1, \ldots, r$, 
then we write $x \xRightarrow{(\lambda, -, q=0)}_{\raisebox{1ex}{\scriptsize{$\sigma$}}} y$. 
\end{enu}
If $\sigma = 1$, then we omit $\sigma$, and write as $x \xRightarrow{(\lambda, \pm)} y$ or $x \xRightarrow{(\lambda, \pm, q=0)} y$. 
\end{definition}

\subsection{Definition of interpolated QLS paths}

\begin{definition} \label{def:iQLS}
A triple $(\underline{x}; \underline{y}; \underline{\sigma})$ of 
sequences $\underline{x}: x_{1}, \ldots, x_{s}$ of elements in $W$ such that $x_{i} \not= x_{i+1}$ for all $1 \le i \le s-1$, 
$\underline{y}: y_{1}, \ldots, y_{s-1}$ of elements in $W$ such that $y_{i} \not= y_{i+1}$ for all $1 \le i \le s-2$, 
and $\underline{\sigma}: \sigma_{0}, \ldots, \sigma_{s}$ of rational numbers, 
is called an \textit{interpolated quantum Lakshmibai-Seshadri} (QLS) \emph{path} of shape $\lambda$ 
if the following conditions hold: 
\begin{enu}
\item $0 = \sigma_{0} < \sigma_{1} < \cdots < \sigma_{s} = 1$; 
\item for all $1 \le i \le s-1$, we have $x_{i+1} \xRightarrow{(\lambda, -)}_{\raisebox{1ex}{\scriptsize{$\sigma_{i}$}}} y_{i}$; 
\item for all $1 \le i \le s-1$, we have $y_{i} \xRightarrow{(\lambda, +)}_{\raisebox{1ex}{\scriptsize{$\sigma_{i}$}}} x_{i}$. 
\end{enu}
Let $\IQLS(\lambda)$ denote the set of all interpolated QLS paths of shape $\lambda$. 
\end{definition}

For $\eta = (\underline{x}; \underline{y}; \underline{\sigma}) \in \IQLS(\lambda)$, 
a sequence $\underline{y}$ may be empty; 
in such a case, we write $\eta = (\underline{x}; ; \underline{\sigma})$. 

\begin{remark} \label{rem:IQLS=QLS}
Let $\lambda \in P^{+}$ be a \emph{regular} dominant weight, i.e., $\pair{\lambda}{\alpha_{i}^{\vee}} > 0$ for all $i \in I$. 
Then we can identify $\IQLS(\lambda)$ with $\QLS(\lambda)$ as follows. 
Since $\Dn{\lambda} = \emptyset$, all paths of the form $w \xRightarrow{(\lambda, -)}_{\raisebox{1ex}{\scriptsize{$\sigma$}}} u$, with $w, u \in W$ and $\sigma \in \BQ$, are the trivial one. 
Hence $\IQLS(\lambda)$ is the set of all elements $\eta = (x_{1}, \ldots, x_{s}; x_{2}, \ldots, x_{s}; \sigma_{0}, \ldots, \sigma_{s})$, with $x_{1}, \ldots, x_{s} \in W$ and $\sigma_{0}, \ldots, \sigma_{s} \in \BQ$, satisfying: 
\begin{enu}
\item $0 = \sigma_{0} < \sigma_{1} < \cdots < \sigma_{s} = 1$, 
\item for all $1 \le i \le s-1$, $x_{i+1} \xRightarrow{(\lambda, +)}_{\raisebox{1ex}{\scriptsize{$\sigma_{i}$}}} x_{i}$. 
\end{enu}
In addition, since $\Dp{\lambda} = \Delta^{+}$, we can use all the positive roots as labels of paths of the form $w \xRightarrow{(\lambda, +)}_{\raisebox{1ex}{\scriptsize{$\sigma$}}} u$, with $w, u \in W$ and $\sigma \in \BQ$. 
Therefore, the path $\eta' := (x_{1}, \ldots, x_{s}; \sigma_{0}, \ldots, \sigma_{s})$ is a QLS path, and hence we obtain a bijection $\IQLS(\lambda) \xrightarrow{\sim} \QLS(\lambda)$, given by $\eta \mapsto \eta'$. 
\end{remark}

\begin{remark} \label{rem:minuscule case}
If $\lambda \in P$ is a \emph{minuscule} weight, i.e., $\pair{\lambda}{\alpha^{\vee}} \in \{-1, 0, 1\}$ for all $\alpha \in \Delta$, then we see that 
\begin{equation}
\IQLS(\lambda) = \{ (w; ; 0, 1) \mid w \in W \}; 
\end{equation}
that is, $\IQLS(\lambda)$ consists only of the paths of the form $(w; ; 0, 1)$, with $w \in W$, i.e., straight-line paths. 
Indeed, since $\pair{\lambda}{\alpha^{\vee}} = \pm 1$ for all $\alpha \in \Dp{\lambda} \sqcup \Dn{\lambda}$, 
there is no $\sigma \in \BQ$, with $0 < \sigma < 1$, such that $\sigma \pair{\lambda}{\alpha^{\vee}} \in \BZ$. 
\end{remark}

As for an (ordinary) QLS path, we can define the initial (or final) direction and the weight for an interpolated QLS path. 

\begin{definition}
Let $\eta = (x_{1}, \ldots, x_{s}; y_{1}, \ldots, y_{s-1}; \sigma_{0}, \ldots, \sigma_{s}) \in \IQLS(\lambda)$ 
be an interpolated QLS path of shape $\lambda$. 
\begin{enu}
\item We set $\iota(\eta) := x_{1}$ and $\kappa(\eta) := x_{s}$; 
we call $\iota(\eta)$ the \textit{initial direction} of $\eta$, 
and $\kappa(\eta)$ the \textit{final direction} of $\eta$. 
\item We set 
\begin{equation}
\wt(\eta) := \sum_{k=1}^{s} (\sigma_{k} - \sigma_{k-1}) x_{k} \lambda, 
\end{equation}
and call it the \textit{weight} of $\eta$. 
\end{enu}
\end{definition}

Also, we define the negativity length of an interpolated QLS path, 
which is not defined for an (ordinary) QLS path. 

\begin{definition}
Let $\eta = (x_{1}, \ldots, x_{s}; y_{1}, \ldots, y_{s-1}; \sigma_{0}, \ldots, \sigma_{s}) \in \IQLS(\lambda)$ 
be an interpolated QLS path of shape $\lambda$. 
We set 
\begin{equation}
\nega(\eta) := \sum_{k=1}^{s-1} \ell(x_{k+1} \Rightarrow y_{k}), 
\end{equation}
and call it the \textit{negativity length} of $\eta$. 
\end{definition}

%%%%%%%%%%%%%%%
% ex:IQLS_n=2 %
%%%%%%%%%%%%%%%

\begin{example} \label{ex:IQLS_n=2}
Assume that $\Fg$ is of type $A_{2}$. 
Let $\lambda = -\vpi_{1} + 2\vpi_{2}$. 
Let us consider the set $\IQLS(\lambda)$. 
Note that $\pair{\lambda}{\alpha_{1}^{\vee}} = -1$, $\pair{\lambda}{\alpha_{1}^{\vee}+\alpha_{2}^{\vee}} = 1$, and $\pair{\lambda}{\alpha_{2}^{\vee}} = 2$. 
Hence $\alpha_{1} \in \Dn{\lambda}$ and $\alpha_{1}+\alpha_{2}, \alpha_{2} \in \Dp{\lambda}$. 
If we define a total order $\vtl$ on $\Delta^{+}$ 
by: $\alpha_{1} \vtl \alpha_{1} + \alpha_{2} \vtl \alpha_{2}$, 
then we have ${\vtl} \in \RO(\lambda, \Delta^{+})$. 

First, we see that $\IQLS(\lambda)$ contains all straight-line paths $(x; ; 0, 1)$, $x \in W$. 
Let us consider the other paths. 
Observe that $\alpha_{2}$ is the unique positive root $\alpha \in \Delta^{+}$ such that $\sigma \pair{\lambda}{\alpha^{\vee}} \in \BZ$ for some $\sigma \in \BQ \cap (0, 1)$; note that we necessarily have $\sigma = 1/2$. 
Hence the paths $\eta \in \IQLS(\lambda)$ which are not straight-line paths are of the form 
$(xs_{2}, x; x; 0, 1/2, 1)$ for some $x \in W$ such that $x \xrightarrow{\alpha_{2}} x s_{2}$. 
Since $\alpha_{2}$ is a simple root, we have the directed edge $x \xrightarrow{\alpha_{2}} xs_{2}$ for all $x \in W$. 
Therefore, we conclude that 
\begin{equation}
\IQLS(\lambda) = \{ (x; ; 0, 1) \mid x \in W \} \sqcup \{ (xs_{2}, x; x; 0, 1/2, 1) \mid x \in W \}. 
\end{equation}
For $x \in W$, we have $\wt((x; ; 0, 1)) = x\lambda = x(-\vpi_{1}+2\vpi_{2})$. 
Also, since $s_{2}\lambda = -\lambda$, we have for $x \in W$, 
\begin{equation}
\wt((xs_{2}, x; x; 0, 1/2, 1)) = \frac{1}{2}xs_{2}\lambda + \frac{1}{2}x\lambda = -\frac{1}{2}x\lambda + \frac{1}{2}x\lambda = 0. 
\end{equation}

\end{example}

\begin{example} \label{ex:IQLS_n=2_2}
Assume that $\Fg$ is of type $A_{2}$. Let $\lambda = -\vpi_{1} + 3\vpi_{2}$. 
We consider the set $\IQLS(\lambda)$. 
Note that $\pair{\lambda}{\alpha_{1}^{\vee}} = -1$, $\pair{\lambda}{\alpha_{1}^{\vee} + \alpha_{2}^{\vee}} = 2$, and $\pair{\lambda}{\alpha_{2}^{\vee}} = 3$. Hence $\alpha_{1} \in \Dn{\lambda}$ and $\alpha_{1}+\alpha_{2}, \alpha_{2} \in \Dp{\lambda}$. 
If we define a total order $\vtl$ on $\Delta^{+}$ by: $\alpha_{1} \vtl \alpha_{1}+\alpha_{2} \vtl \alpha_{2}$, then we have ${\vtl} \in \RO(\lambda, \Delta^{+})$. 

As in Example~\ref{ex:IQLS_n=2}, we see that $\IQLS(\lambda)$ contains all straight-line paths $(x; ; 0, 1)$, $x \in W$. 
Let us consider the other paths; for each $\gamma \in \Delta^{+}$, we need to determine those $\sigma \in \BQ \cap (0, 1)$ for which $\sigma \pair{\lambda}{\gamma^{\vee}} \in \BZ$. 
If $\gamma = \alpha_{1}$, then there is no such $\sigma$. 
For $\gamma = \alpha_{1}+\alpha_{2}$, we see that $\sigma = 1/2$. 
If $\gamma = \alpha_{2}$, then $\sigma = 1/3, 2/3$. 

Now, we consider $\eta = (x_{1}, \ldots, x_{s}; y_{1}, \ldots, y_{s-1}; \sigma_{0}, \ldots, \sigma_{s}) \in \IQLS(\lambda)$, with $s \ge 2$. 
Let us consider the case that $\sigma_{s-1} = 1/3$. In this case, there is no $\sigma \in \BQ \cap (0,1)$, with $\sigma < 1/3$, such that $\sigma \pair{\lambda}{\gamma^{\vee}} \in \BZ$ for some $\gamma \in \Delta^{+}$. Hence it follows that $s = 2$. 
Then, we have $y_{1} = x_{2}$ since there is no $(1/3)$-path of the form $x \xRightarrow{(\lambda, -)}_{\raisebox{1ex}{\scriptsize{$1/3$}}} y$. Also, we see that $w_{1} \xrightarrow{\gamma_{1}} w_{2} \xrightarrow{\gamma_{2}} \cdots \xrightarrow{\gamma_{r}} w_{r+1}$ is a $(1/3)$-path if and only if $w_{1} \in W$, $r = 1$, $\gamma_{1} = \alpha_{2}$, and $w_{2} = w_{1}s_{2}$. 
Therefore, we deduce that 
\begin{equation}
\eta = (xs_{2}, x; x; 0,1/3,1), \quad x \in W. 
\end{equation}

Next, consider the case that $\sigma_{s-1} = 1/2$. In this case, $s = 2, 3$; if $s = 3$, then $\sigma_{s-2} = \sigma_{1} = 1/3$. 
Assume that $s = 2$. Then, by the same argument as above, we see that $y_{1} = x_{2}$ and $y_{1} \xrightarrow{\alpha_{1}+\alpha_{2}} x_{1}$. 
Hence we deduce that 
\begin{equation}
\eta = \underbrace{(s_{1}s_{2}, s_{2}; s_{2}; 0, 1/2, 1)}_{=: \, \eta_{1}}, \underbrace{(s_{2}s_{1}, s_{1}; s_{1}; 0, 1/2, 1)}_{=: \, \eta_{2}}, \underbrace{(e, s_{1}s_{2}s_{1}; s_{1}s_{2}s_{1}; 0, 1/2, 1)}_{=: \, \eta_{3}}. 
\end{equation}
If $s = 3$, then $y_{1} = x_{2}$ and $y_{1} \xrightarrow{\alpha_{2}} x_{1}$. 
By combining this fact with the result for $s = 2$, we deduce that 
\begin{equation}
\begin{split}
\eta &= \underbrace{(s_{1}, s_{1}s_{2}, s_{2}; s_{1}s_{2}, s_{2}; 0, 1/3, 1/2, 1)}_{=: \, \eta_{4}}, \underbrace{(s_{1}s_{2}s_{1}, s_{2}s_{1}, s_{1}; 0, 1/3, 1/2, 1)}_{=: \, \eta_{5}}, \\ 
& \quad \ \underbrace{(s_{2}, e, s_{1}s_{2}s_{1}; e, s_{1}s_{2}s_{1}; 0, 1/3, 1/2, 1)}_{=: \, \eta_{6}}. 
\end{split}
\end{equation}

Let us consider the remaining case that $\sigma_{s-1} = 2/3$; we necessarily have $s = 2, 3, 4$. 
Assume that $s = 2$. Then, by the same argument as for $\sigma_{s-1} = 1/3$, we see that 
\begin{equation}
\eta = (xs_{2}, x; x; 0,2/3,1), \quad x \in W. 
\end{equation}
Assume that $s = 3$. Then, we have $\sigma_{s-2} = \sigma_{1} = 1/3, 1/2$; note that $y_{1} = x_{2}$ in both cases. 
If $\sigma_{1} = 1/3$, then $y_{1} \xrightarrow{\alpha_{2}} x_{1}$. 
Hence we see that 
\begin{equation}
\eta = (x, xs_{2}, x; xs_{2}, x; 0, 1/3, 2/3, 1), \quad x \in W. 
\end{equation}
If $\sigma_{1} = 2/3$, then the argument for $\sigma_{s-1} = 1/2$ shows that 
\begin{equation}
(y_{1}, x_{1}) = (s_{2}, s_{1}s_{2}), (s_{1}, s_{2}s_{1}), (s_{1}s_{2}s_{1}, e). 
\end{equation}
Therefore, we deduce that 
\begin{equation}
\begin{split}
\eta &= \underbrace{(s_{1}s_{2}, s_{2}, e; s_{2}, e; 0, 1/2, 2/3, 1)}_{=: \, \eta_{7}}, \underbrace{(s_{2}s_{1}, s_{1}, s_{1}s_{2}; s_{1}, s_{1}s_{2}; 0, 1/2, 2/3, 1)}_{=: \, \eta_{8}}, \\ 
& \quad \ \underbrace{(e, s_{1}s_{2}s_{1}, s_{2}s_{1}; s_{1}s_{2}s_{1}, s_{2}s_{1}; 0, 1/2, 2/3, 1)}_{=: \, \eta_{9}}. 
\end{split}
\end{equation}
Finally, we consider the case $s = 4$. Then, $\sigma_{s-2} = \sigma_{2} = 1/2$ and $\sigma_{s-3} = \sigma_{1} = 1/3$. 
By combining the result for $s = 3$ and the one for $\sigma_{s-1} = 1/3$, we deduce that 
\begin{equation}
\begin{split}
\eta &= \underbrace{(s_{1}, s_{1}s_{2}, s_{2}, e; s_{1}s_{2}, s_{2}, e; 0, 1/3, 1/2, 2/3, 1)}_{=: \, \eta_{10}}, \\ 
& \quad \ \underbrace{(s_{1}s_{2}s_{1}, s_{2}s_{1}, s_{1}, s_{1}s_{2}; s_{2}s_{1}, s_{1}, s_{1}s_{2}; 0, 1/3, 1/2, 2/3, 1)}_{=: \, \eta_{11}}, \\ 
& \quad \ \underbrace{(s_{2}, e, s_{1}s_{2}s_{1}, s_{2}s_{1}; e, s_{1}s_{2}s_{1}, s_{2}s_{1}; 0, 1/3, 1/2, 2/3, 1)}_{=: \, \eta_{12}}. 
\end{split}
\end{equation}

Thus, we conclude that 
\begin{equation}
\begin{split}
\IQLS(\lambda) &= \{(x; ; 0, 1) \mid x \in W\} \\ 
& \quad \sqcup \{(xs_{2}, x; x; 0, 1/3, 1) \mid x \in W\} \\ 
& \quad \sqcup \{(xs_{2}, x; x; 0, 2/3, 1) \mid x \in W\} \\ 
& \quad \sqcup \{(x, xs_{2}, x; xs_{2}, x; 0, 1/3, 2/3, 1) \mid x \in W\} \\ 
& \quad \sqcup \{\eta_{1}, \eta_{2}, \ldots, \eta_{12}\}. 
\end{split}
\end{equation}
\end{example}

\subsection{Interpolated LS paths}

We also introduce interpolated Lakshmibai-Seshadri (LS) paths, 
which are the ``$q = 0$ version'' of interpolated QLS paths. 
Interpolated LS paths generalize (ordinary) LS paths. 

\begin{definition}
A triple $(\underline{x}; \underline{y}; \underline{\sigma})$ of sequences $\underline{x}: x_{1}, \ldots, x_{s}$ of elements in $W$, 
$\underline{y}: y_{1}, \ldots, y_{s-1}$ of elements in $W$, 
and $\underline{\sigma}: \sigma_{0}, \ldots, \sigma_{s}$ of rational numbers, 
is called an \textit{interpolated Lakshmibai-Seshadri} (LS) \emph{path} of shape $\lambda$ 
if the following conditions hold: 
\begin{enu}
\item $0 = \sigma_{0} < \sigma_{1} < \cdots < \sigma_{s} = 1$; 
\item for all $1 \le i \le s-1$, we have $x_{i+1} \xRightarrow{(\lambda, -, q=0)}_{\raisebox{1ex}{\scriptsize{$\sigma_{i}$}}} y_{i}$; 
\item for all $1 \le i \le s-1$, we have $y_{i} \xRightarrow{(\lambda, +, q=0)}_{\raisebox{1ex}{\scriptsize{$\sigma_{i}$}}} x_{i}$. 
\end{enu}
Let $\ILS(\lambda)$ denote the set of all interpolated LS paths of shape $\lambda$. 
\end{definition}

By the definition of interpolated LS paths, it is obvious that $\ILS(\lambda) \subset \IQLS(\lambda)$. 

\begin{example}
Assume that $\Fg$ is of type $A_{2}$, and let $\lambda = -\vpi_{1} + 2\vpi_{2}$. 
Recall Example~\ref{ex:IQLS_n=2}. 
It is obvious that $(x; ; 0, 1) \in \ILS(\lambda)$ for all $x \in W$. 
For $x \in W$, we see that $(xs_{2}, x; x; 0, 1/2, 1) \in \ILS(\lambda)$ 
if and only if $x \xrightarrow{\alpha_{2}} xs_{2}$ is a Bruhat edge. 
Since $x \xrightarrow{\alpha_{2}} xs_{2}$ is a Bruhat edge for $x = e, s_{1}, s_{2}s_{1}$, 
and it is a quantum edge for $x = s_{2}, s_{1}s_{2}, s_{1}s_{2}s_{1}$, 
we conclude that 
\begin{equation}
\ILS(\lambda) = \{ (x; ; 0, 1) \mid x \in W \} \sqcup \{ (xs_{2}, x; x; 0, 1/2, 1) \mid x = e, s_{1}, s_{2}s_{1} \}. 
\end{equation}
\end{example}

\begin{example}
Assume that $\Fg$ is of type $A_{2}$. Let $\lambda = -\vpi_{1} + 3\vpi_{2}$. From Example~\ref{ex:IQLS_n=2_2}, we can verify that 
\begin{equation}
\begin{split}
\ILS(\lambda) &= \{(x; ; 0,1) \mid x \in W\} \\ 
& \quad \sqcup \{(xs_{2}, x; x; 0, 1/3, 1) \mid x = e, s_{1}, s_{2}s_{1}\} \\ 
& \quad \sqcup \{(xs_{2}, x; x; 0, 2/3, 1) \mid x = e, s_{1}, s_{2}s_{1}\} \\ 
& \quad \sqcup \{\eta_{1}, \eta_{2}, \eta_{5}, \eta_{7}\}. 
\end{split}
\end{equation}
\end{example}

%%%%%%%%%%%%%%
% sec:forgetful %
%%%%%%%%%%%%%%

\section{The forgetful map} \label{sec:forgetful}

We construct a map (called the forgetful map) from $\CA(w, \Gamma)$ to $\IQLS(\lambda)$ 
for $w \in W$ and $\lambda \in P$, where $\Gamma$ is the suitable $\lambda$-chain corresponding to a fixed ${\vtl} \in \RO(\lambda, \Delta^{+})$. 
Note that our construction of the forgetful map is a generalization of the construction of the forgetful map for dominant integral weights explained in \cite{LNSSS2}. 
Also, the proof of properties of the forgetful map are similar to those in \cite{NNS}. 
In this section, we take and fix an arbitrary $\lambda \in P$. 

\subsection{Construction of the forgetful map} 

Let us fix ${\vtl} \in \RO(\lambda, \Delta^{+})$, and take the corresponding suitable $\lambda$-chain $\Gamma_{\vtl}(\lambda)$, 
or equivalently, the suitable reduced expression $t_{\lambda} = s_{i_{r}} \cdots s_{i_{2}} s_{i_{1}} (\pi^\vee)^{-1}$ for $\vtl$. 
We define a map from $\CA(w, \Gamma_{\vtl}(\lambda))$ to $\IQLS(\lambda)$ for each $w \in W$. 
Let $A = \{j_{1}, \ldots, j_{s}\} \in \CA(w, \Gamma_{\vtl}(\lambda))$. Then, by the definition of admissible subsets, we obtain the following directed path in $\QBG(W)$: 
\begin{equation}
\bp(A): w = u_{0} \xrightarrow{|\gamma_{j_{1}}^{\sL}|} u_{1} \xrightarrow{|\gamma_{j_{2}}^{\sL}|} \cdots \xrightarrow{|\gamma_{j_{s}}^{\sL}|} u_{s}, 
\end{equation}
where $u_{a} = ws_{|\gamma_{j_{1}}^{\sL}|} \cdots s_{|\gamma_{j_{a}}^{\sL}|}$ for $a = 0, \ldots, s$. 
We set $d_{i} := \deg((\beta_{i}^{\sL})^{\vee})/\pair{\lambda}{\overline{(\beta_{i}^{\sL})^{\vee}}}$ for $i = 1, \ldots, r$. 

Since $\beta_{1}^{\sL} < \cdots < \beta_{r}^{\sL}$, we see from definition of the order $<$ that there exists $m_{1} < \cdots < m_{t}$ such that 
\begin{align}
0 &= d_{j_{1}} = \cdots = d_{j_{m_{1}}} \\
&< d_{j_{m_{1}+1}} = \cdots = d_{j_{m_{2}}} \\
&< \cdots \\
&< d_{j_{m_{t}+1}} = \cdots = d_{j_{s}} = 1. 
\end{align}
If $d_{j_{1}} > 0$, then we set $m_{1} := 0$. 
Also, if $d_{j_{s}} < 1$, then we set $m_{t} := s$. 

For each $a = 1, \ldots, t$, we consider the directed path in $\QBG(W)$: 
\begin{equation}
u_{m_{a}} \xrightarrow{|\gamma_{j_{m_{a}+1}}^{\sL}|} u_{m_{a}+1} \xrightarrow{|\gamma_{j_{m_{a}+2}}^{\sL}|} \cdots \xrightarrow{|\gamma_{j_{m_{a+1}}}^{\sL}|} u_{m_{a+1}}. 
\end{equation}
By the definition of $<$ on $\Daf{\lambda}$ and the assumption that $\beta_{1}^{\sL} < \cdots < \beta_{r}^{\sL}$, 
we have $\gamma_{m_{a}+1}^{\sL} \succ \gamma_{m_{a}+2}^{\sL} \succ \cdots \succ \gamma_{m_{a+1}}^{\sL}$. 
Therefore, we have the following cases: 
\begin{enu}
\item $\gamma_{p}^{\sL} \in \Dp{\lambda}$ for all $m_{a}+1 \le p \le m_{a+1}$ (in this case, we set $n_{a} := m_{a}$); 
\item there exists $n_{a} \in \{m_{a}+1, \ldots, m_{a+1}-1\}$ such that $\gamma_{p}^{\sL} \in -\Dn{\lambda}$ for all $m_{a}+1 \le p \le n_{a}$, 
and $\gamma_{p}^{\sL} \in \Dp{\lambda}$ for all $n_{a} < p \le m_{a+1}$; 
\item $\gamma_{p}^{\sL} \in -\Dn{\lambda}$ for all $m_{a}+1 \le p \le m_{a+1}$ (in this case, we set $n_{a} := m_{a+1}$). 
\end{enu}

We wet 
\begin{equation}
\begin{split}
\Xi(A) &:= (u_{m_{t}}, u_{m_{t-1}}, \ldots, u_{m_{1}}; u_{n_{t-1}}, u_{n_{t-2}}, \ldots, u_{n_{1}}; \\ 
& \hspace{10mm} 0, 1-d_{j_{m_{t}}}, 1-d_{j_{m_{t-1}}}, \ldots, 1-d_{j_{m_{2}}}, 1). 
\end{split}
\end{equation}
The following lemma shows that $\Xi$ defines a map $\CA(w, \Gamma_{\vtl}(\lambda)) \rightarrow \IQLS(\lambda)$. 

\begin{lemma}
For $A \in \CA(w, \Gamma_{\vtl}(\lambda))$, we have $\Xi(A) \in \IQLS(\lambda)$. 
\end{lemma}
\begin{proof}
Let $A = \{j_{1}, \ldots, j_{t}\}$. 
It suffices to show that $(1-d_{j_{m_{a}}}) \pair{\lambda}{|\gamma_{b}^{\sL}|^{\vee}} \in \BZ$ for all $2 \le a \le t$ and $m_{a-1}+1 \le b \le m_{a}$. 
This follows from the following computation: 
\begin{align}
(1-d_{j_{m_{a}}}) \pair{\lambda}{|\gamma_{b}^{\sL}|^{\vee}} &= (1-d_{b})\pair{\lambda}{|\gamma_{b}^{\sL}|^{\vee}} \\
&= \frac{\pair{\lambda}{\overline{(\beta_{b}^{\sL})^{\vee}}}-\deg((\beta_{b}^{\sL})^{\vee})}{\pair{\lambda}{\overline{(\beta_{b}^{\sL})^{\vee}}}} \cdot \pair{\lambda}{|\overline{(\beta_{b}^{\sL})^{\vee}}|} \\
&= \pm (\pair{\lambda}{\overline{(\beta_{b}^{\sL})^{\vee}}}-\deg((\beta_{b}^{\sL})^{\vee})) \in \BZ. 
\end{align}
Thus we conclude that $\Xi(A) \in \IQLS(\lambda)$. This proves the lemma. 
\end{proof}

Finally, we define the forgetful map, which is a generalization of the map $\Pi^{*}$ in \cite[Section~6.1]{LNSSS2} and the map $\Xi$ in \cite[Section~3.3]{NNS}. 

%%%%%%%%%%%%%%
% def:forgetful %
%%%%%%%%%%%%%%

\begin{definition} \label{def:forgetful}
The \textit{forgetful map} $\widetilde{\Xi}: \CA(w, \Gamma_{\vtl}(\lambda)) \rightarrow \IQLS(\lambda) \times W$ 
is defined by $\widetilde{\Xi}(A) := (\Xi(A), \ed(A))$ for $A \in \CA(w, \Gamma_{\vtl}(\lambda))$. 
\end{definition}

\begin{example} \label{ex:forgetful}
Assume that $\Fg$ is of type $A_{2}$. Let $\lambda = -\vpi_{1} + 2\vpi_{2}$, 
and take ${\vtl} \in \RO(\lambda, \Delta^{+})$ given by: $\alpha_{1} \vtl \alpha_{1}+\alpha_{2} \vtl \alpha_{2}$. 
Then we have 
\begin{equation}
(\beta_{1}^{\sL})^{\vee} = \alpha_{2}^{\vee}, \quad (\beta_{2}^{\sL})^{\vee} = \alpha_{1}^{\vee}+\alpha_{2}^{\vee}, \quad (\beta_{3}^{\sL})^{\vee} = \alpha_{2}^{\vee}+\wti{\delta}, \quad (\beta_{4}^{\sL})^{\vee} = -\alpha_{1}^{\vee} + \wti{\delta}, 
\end{equation}
with $\Daf{-\vpi_{1}+2\vpi_{2}} = \{(\beta_{1}^{\sL})^{\vee}, (\beta_{2}^{\sL})^{\vee}, (\beta_{3}^{\sL})^{\vee}, (\beta_{4}^{\sL})^{\vee} \}$, and 
\begin{equation}
\Gamma_{\vtl}(\lambda) = (\alpha_{2}, \alpha_{1}+\alpha_{2}, \alpha_{2}, -\alpha_{1}). 
\end{equation}
Let us compute $\Xi(A)$ for $A = \{2, 3, 4\} \in \CA(s_{1}, \Gamma_{\vtl}(\lambda))$. 
First, we see that 
\begin{equation}
\bp(A): s_{1} \xrightarrow{\alpha_{1}+\alpha_{2}} s_{2}s_{1} \xrightarrow{\alpha_{2}} s_{1}s_{2}s_{1} \xrightarrow{\alpha_{1}} s_{1}s_{2}; 
\end{equation}
note that $d_{2} = 0$, $d_{3} = 1/2$, and $d_{4} = 1$. 
Therefore, in the definition of $\Xi(A)$, we have $m_{1} = 1$ and $m_{t} = 2$ (with $t = 2$). 
Let us concentrate on the directed path $s_{2}s_{1} (= u_{1}) \xrightarrow{\alpha_{2}} s_{1}s_{2}s_{1} (= u_{2})$ in $\bp(A)$; 
note that $\alpha_{2} \in \Dp{\lambda}$. 
Thus we deduce that $n_{1} = m_{1} = 1$, and hence that
\begin{equation}
\Xi(A) = (u_{2}, u_{1}; u_{1}; 0, 1-d_{2}, 1) = (s_{1}s_{2}s_{1}, s_{2}s_{1}; s_{2}s_{1}; 0, 1/2, 1). 
\end{equation}

In Table~\ref{tab:forgetful}, we give the list of $\tXi(A)$ for all $A \in \CA(w, \Gamma_{\vtl}(\lambda))$. 
\begin{table}[ht]
\centering
\caption{The list of $\tXi(A)$ for all $A \in \CA(s_{1}, \Gamma_{\vtl}(\lambda))$}
\label{tab:forgetful}
\begin{tabular}{|c||c|} \hline
$A$ & $\tXi(A) = (\Xi(A), \ed(A))$ \\ \hline 
$\emptyset$ & $((s_{1}; ; 0, 1), s_{1})$  \\ 
$\{1\}$ & $((s_{1}s_{2}; ; 0,1), s_{1}s_{2})$ \\ 
$\{2\}$ & $((s_{2}s_{1}; ; 0, 1), s_{2}s_{1})$ \\ 
$\{3\}$ & $((s_{1}s_{2}, s_{1}; s_{1}; 0, 1/2, 1), s_{1}s_{2})$ \\
$\{4\}$ & $((s_{1}; ; 0, 1), e)$ \\ 
$\{1, 3\}$ & $((s_{1}, s_{1}s_{2}; s_{1}s_{2}; 0, 1/2, 1), s_{1})$ \\ 
$\{1, 4\}$ & $((s_{1}s_{2}; ; 0, 1), s_{1}s_{2}s_{1})$ \\ 
$\{2, 3\}$ & $((s_{1}s_{2}s_{1}, s_{2}s_{1}; s_{2}s_{1}; 0, 1/2, 1), s_{1}s_{2}s_{1})$ \\ 
$\{2, 4\}$ & $((s_{2}s_{1}; ; 0, 1), s_{2})$ \\ 
$\{3, 4\}$ & $((s_{1}s_{2}, s_{1}; s_{1}; 0, 1/2, 1), s_{1}s_{2}s_{1})$ \\ 
$\{1, 3, 4\}$ & $((s_{1}, s_{1}s_{2}; s_{1}s_{2}; 0, 1/2, 1), e)$ \\ 
$\{2, 3, 4\}$ & $((s_{1}s_{2}s_{1}, s_{2}s_{1}; s_{2}s_{1}; 0, 1/2, 1), s_{1}s_{2})$ \\ \hline
\end{tabular}
\end{table}
\end{example}

\subsection{The image of the forgetful map}

The image $\im(\tXi)$ of the forgetful map $\tXi$ is described by the following.

%%%%%%%%%%%%%%%%%
% thm:im_forgetful %
%%%%%%%%%%%%%%%%%

\begin{theorem} \label{thm:im_forgetful}
We have the following equality: 
\begin{equation}
\im(\tXi) = \left\{ (\eta, u) \in \IQLS(\lambda) \times W \ \middle| \begin{array}{ll} w \xRightarrow{(\lambda, +)} \kappa(\eta) \\ \iota(\eta) \xRightarrow{(\lambda, -)} u \end{array} \right\}. \label{eq:image_forgetful} 
\end{equation}
\end{theorem}

\begin{remark}
Let $\lambda \in P^{+}$ be a regular dominant weight. 
Under the identification $\IQLS(\lambda)$ and $\QLS(\lambda)$ explained in Remark~\ref{rem:IQLS=QLS}, we see that the map $\Xi$ is identical to the forgetful map constructed in \cite[Proposition~28]{LNS}. 
In addition, the right-hand side of \eqref{eq:image_forgetful} is identified with $\IQLS(\lambda) = \QLS(\lambda)$ as follows: 
Since $\Dp{\lambda} = \Delta^{+}$, there exists a path $w \xRightarrow{(\lambda, +)} \kappa(\eta)$ for all $\eta \in \IQLS(\lambda)$ by the shellability (Theorem~\ref{thm:shellability}). 
Also, since $\Dn{\lambda} = \emptyset$, all paths of the form $\iota(\eta) \xRightarrow{(\lambda, -)} u$ for $\eta \in \IQLS(\lambda)$ and $u \in W$ are trivial. 
Hence 
\begin{equation}
\text{(RHS of \eqref{eq:image_forgetful})} = \{ (\eta, \iota(\eta)) \mid \eta \in \IQLS(\lambda) \} \simeq \IQLS(\lambda) = \QLS(\lambda). 
\end{equation}
Under this identification, $\tXi$ is identical to $\Xi$ and hence to the forgetful map constructed in \cite[Proposition~28]{LNS}.
\end{remark}

\begin{proof}[Proof of Theorem~\ref{thm:im_forgetful}]
Let $\mathcal{S}$ denote the right-hand side of \eqref{eq:image_forgetful}. 
First, we prove the inclusion $\im(\tXi) \subset \CS$. 
Let $A = \{ j_{1}, \ldots, j_{s} \} \in \CA(w, \Gamma_{\vtl}(\lambda))$. 
It suffices to show that $w \xRightarrow{(\lambda, +)} \kappa(\Xi(A))$ and $\iota(\Xi(A)) \xRightarrow{(\lambda, -)} \ed(A)$. 

Recall that $\kappa(\Xi(A)) = u_{m_{1}}$, and that there exists a directed path 
\begin{equation}
w \xrightarrow{|\gamma_{j_{1}}^{\sL}|} u_{1} \xrightarrow{|\gamma_{j_{2}}^{\sL}|} \cdots \xrightarrow{|\gamma_{j_{m_{1}}}^{\sL}|} u_{m_{1}}
\end{equation}
in $\QBG(W)$; 
note that $d_{j_{1}} = \cdots = d_{j_{m_{1}}} = 0$. 
Let $a = j_{1}, j_{2}, \ldots, j_{m_{1}}$. Then, we see that $\deg((\beta_{a}^{\sL})^{\vee}) = 0$ since 
\begin{equation}
d_{a} = \frac{\deg((\beta_{a}^{\sL})^{\vee})}{\pair{\lambda}{\overline{(\beta_{a}^{\sL})^{\vee}}}} = 0. 
\end{equation}
By \eqref{eq:Inv}, we have $\gamma_{a}^{\sL} = \overline{\beta_{a}^{\sL}} \in \Delta^{+}$. 
This implies that $|\gamma_{a}^{\sL}| = \gamma_{a}^{\sL} \in \Dp{\lambda}$. 
Since $\gamma_{j_{1}}^{\sL} \succ \cdots \succ \gamma_{j_{m_{1}}}^{\sL}$, 
it follows that $\gamma_{j_{1}}^{\sL} \vtr \cdots \vtr \gamma_{j_{m_{1}}}^{\sL}$. 
Therefore, we deduce that $w \xRightarrow{(\lambda, +)} \kappa(\Xi(A))$, as desired. 

Next, we show that $\iota(\Xi(A)) \xRightarrow{(\lambda, -)} \ed(A)$. 
Recall that $\iota(\Xi(A)) = u_{m_{t}}$, and that there exists a directed path 
\begin{equation}
u_{m_{t}} \xrightarrow{|\gamma_{j_{m_{t}}+1}^{\sL}|} u_{m_{t}+1} \xrightarrow{|\gamma_{j_{m_{t}}+2}^{\sL}|} \cdots \xrightarrow{|\gamma_{j_{s}}^{\sL}|} u_{s} = \ed(A) 
\end{equation}
in $\QBG(W)$; 
note that $d_{j_{m_{t}+1}} = \cdots = d_{j_{s}} = 1$. 
Let $a = j_{m_{t}+1}, j_{m_{t}+2}, \ldots, j_{s}$. 
Then, we see that $\deg((\beta_{a}^{\sL})^{\vee}) = \pair{\lambda}{\overline{(\beta_{a}^{\sL})^{\vee}}}$ since 
\begin{equation}
d_{a} = \frac{\deg((\beta_{a}^{\sL})^{\vee})}{\pair{\lambda}{\overline{(\beta_{a}^{\sL})^{\vee}}}} = 1. 
\end{equation}
Here, \eqref{eq:Inv} implies that $\gamma_{a}^{\sL} = \overline{\beta_{a}^{\sL}} \in \Delta^{-}$, 
and hence that $|\gamma_{a}^{\sL}| = -\gamma_{a}^{\sL} \in \Dn{\lambda}$. 
Since $\gamma_{j_{m_{t}}+1}^{\sL} \succ \cdots \succ \gamma_{j_{s}}^{\sL}$, 
it follows that $|\gamma_{j_{m_{t}}+1}^{\sL}| \vtr \cdots \vtr |\gamma_{j_{s}}^{\sL}|$. 
Therefore, we deduce that $\iota(\Xi(A)) \xRightarrow{(\lambda, -)} \ed(A)$, as desired. 
This proves the inclusion $\im(\tXi) \subset \CS$. 

We will prove the opposite inclusion $\CS \subset \im(\tXi)$. 
Let $(\eta, u) \in \CS$. Then we have $w \xRightarrow{(\lambda, +)} \kappa(\eta)$ and $\iota(\eta) \xRightarrow{(\lambda, -)} u$. 
We will construct $A \in \CA(w, \Gamma_{\vtl}(\lambda))$ such that $\Xi(A) = \eta$ and $\ed(A) = u$. 
We write $\eta$ as $(x_{1}, \ldots, x_{s}; y_{1}, \ldots, y_{s-1}; \sigma_{0}, \ldots, \sigma_{s})$. 
Since $w \xRightarrow{(\lambda, +)} \kappa(\eta) = x_{s}$, 
there exists a directed path 
\begin{equation} \label{eq:inv_forgetful1}
u \xrightarrow{|\gamma_{j_{s, 1}}^{\sL}|} \cdots \xrightarrow{|\gamma_{j_{s, m_{s}}}^{\sL}|} x_{s}
\end{equation}
in $\QBG(W)$ such that $\gamma_{j_{s,1}}^{\sL} \succ \cdots \succ \gamma_{j_{s,m_{s}}}^{\sL}$, and $\gamma_{j_{s, p}}^{\sL} \in \Dp{\lambda}$ for all $p = 1, \ldots, m_{s}$. 
Also, by the definition of interpolated QLS paths, for each $i = 1, \ldots, s-1$, we have the following directed paths 
\begin{equation} \label{eq:inv_forgetful2}
x_{i+1} \xrightarrow{|\gamma_{j_{i, 1}}^{\sL}|} \cdots \xrightarrow{|\gamma_{j_{i, n_{i}}}^{\sL}|} y_{i}, 
\end{equation}
\begin{equation} \label{eq:inv_forgetful3}
y_{i} \xrightarrow{|\gamma_{j_{i, n_{i}+1}}^{\sL}|} \cdots \xrightarrow{|\gamma_{j_{i, m_{i}}}^{\sL}|} x_{i}
\end{equation}
in $\QBG(W)$ such that $\gamma_{j_{i,1}}^{\sL} \succ \cdots \succ \gamma_{j_{i, m_{i}}}^{\sL}$, and such that $\gamma_{j_{i, p}}^{\sL} \in -\Dn{\lambda}$ for $p = 1, \ldots, n_{i}$, 
$\gamma_{j_{i, p}}^{\sL} \in \Dp{\lambda}$ for $p = n_{i}+1, \ldots, m_{i}$. 
In addition, since $\iota(\eta) = x_{1} \xRightarrow{(\lambda, -)} u$, there exists a directed path 
\begin{equation} \label{eq:inv_forgetful4}
x_{1} \xrightarrow{|\gamma_{j_{0, 1}}^{\sL}|} \cdots \xrightarrow{|\gamma_{j_{0, m_{0}}}^{\sL}|} u 
\end{equation}
in $\QBG(W)$ such that $\gamma_{j_{0,1}}^{\sL} \succ \cdots \succ \gamma_{j_{0,m_{0}}}^{\sL}$, and $\gamma_{j_{0, p}}^{\sL} \in -\Dn{\lambda}$ for $p = 1, \ldots, m_{0}$. 

For $i = 0, \ldots, s$, we set $d_{i} := 1 - \sigma_{i}$. 
Also, for $i = 0, \ldots, s$ and $p = 1, \ldots, m_{i}$, we set 
$\beta_{j_{i, p}}^{\vee} := (\gamma_{j_{i, p}}^{\sL})^{\vee} + d_{i} \pair{\lambda}{(\gamma_{j_{i, p}}^{\sL})^{\vee}} \widetilde{\delta}$. 
We claim that $\beta_{j_{i, p}}^{\vee} \in \Daf{\lambda}$. 

Assume that $i = s$. Then we have $d_{s} = 1 - \sigma_{s} = 0$. 
Since $\gamma_{j_{s, p}}^{\sL} \in \Delta^{+}$, \eqref{eq:Inv} implies that $\beta_{j_{i, p}}^{\vee} \in \Daf{\lambda}$. 
Next, assume that $1 \le i \le s-1$. In this case, we have $0 < \sigma_{i} < 1$. 
Since $0 < d_{i} \pair{\lambda}{(\gamma_{j_{i, p}}^{\sL})^{\vee}} < \pair{\lambda}{(\gamma_{j_{i, p}}^{\sL})^{\vee}}$ (note that $\pair{\lambda}{(\gamma_{j_{i, p}}^{\sL})^{\vee}} > 0$), 
and since $0 \le \chi((\gamma_{j_{i, p}}^{\sL})^{\vee}) \le 1$, 
\eqref{eq:Inv} implies that $\beta_{j_{i, p}}^{\vee} \in \Daf{\lambda}$. 
Finally, assume that $i = 0$. Then we have $d_{0} = 1 - \sigma_{0} = 1$. 
Since $\gamma_{j_{0, p}}^{\sL} \in \Delta^{-}$ (or equivalently, $\chi((\gamma_{j_{0, p}}^{\sL})^{\vee}) = 1$), 
\eqref{eq:Inv} implies that $\beta_{j_{0, p}}^{\vee} \in \Daf{\lambda}$. 
Therefore, we conclude that $\beta_{j_{i, p}}^{\vee} \in \Daf{\lambda}$ for all $i = 0, \ldots, s$ and $p = 1, \ldots, m_{i}$, as desired. 

Now, recall that $\gamma_{j_{i, 1}}^{\sL} \succ \cdots \succ \gamma_{j_{i, m_{i}}}^{\sL}$ for all $i = 0, \ldots, s$, and that 
$0 = d_{s} < \cdots < d_{1} < d_{0} = 1$. 
Also, we have 
\begin{equation}
\Phi(\beta_{j_{i, p}}^{\vee}) = \left( \frac{\deg(\beta_{j_{i, p}}^{\vee})}{\pair{\lambda}{\overline{\beta_{j_{i, p}}^{\vee}}}}, \overline{\beta_{j_{i, p}}} \right) = (d_{i}, \gamma_{j_{i, p}}^{\sL}) 
\end{equation}
for $i = 0, \ldots, s$ and $p = 1, \ldots, m_{i}$. 
Therefore, by the definition of $<$ on $\Daf{\lambda}$, we deduce that 
$\beta_{j_{s, 1}}^{\vee} < \cdots < \beta_{j_{s, m_{s}}}^{\vee} < \beta_{j_{s-1, 1}}^{\vee} < \cdots < \beta_{j_{s-1, m_{s-1}}}^{\vee} 
< \cdots < \beta_{j_{0, 1}}^{\vee} < \cdots < \beta_{j_{0, m_{0}}}^{\vee}$. 
Thus we obtain 
\begin{equation}
A = \{ j'_{s, 1}, \ldots, j'_{s, m_{s}}, j'_{s-1, 1}, \ldots, j'_{s-1, m_{s-1}}, \ldots, j'_{0, 1}, \ldots, j'_{0, m_{0}} \}
\end{equation}
such that $(\beta_{j'_{i, p}}^{\sL})^{\vee} = \beta_{j_{i, p}}^{\vee}$ for all $i = 0, \ldots, s$ and $p = 1, \ldots, m_{i}$. 
It is obvious that $A \in \CA(w, \Gamma_{\vtl}(\lambda))$, and that $\Xi(A) = \eta$, $\ed(A) = u$. 
This proves that $(\eta, u) \in \im(\tXi)$, and hence that $\im(\tXi) = \CS$, as desired. This completes the proof of Theorem~\ref{thm:im_forgetful}. 
\end{proof}

\begin{example}
Assume that $\Fg$ is of type $A_{2}$. For $\lambda = -\vpi_{1}+2\vpi_{2}$ and $w = s_{1}$, we computed $\im(\tXi)$ in Example~\ref{ex:forgetful}. 
Also, recall the set $\IQLS(\lambda)$ from Example~\ref{ex:IQLS_n=2}. 
From Table~\ref{tab:image} below, we see the following for all $\eta \in \IQLS(\lambda)$: 
\begin{itemize}
\item whether $s_{1} \xRightarrow{(\lambda, +)} \kappa(\eta)$ or not, and 
\item which $u \in W$ satisfies $\iota(\eta) \xRightarrow{(\lambda, -)} u$ 
when $s_{1} \xRightarrow{(\lambda, +)} \kappa(\eta)$. 
\end{itemize}
\begin{table}[ht]
\centering
\caption{Conditions on the RHS of equation~\eqref{eq:image_forgetful}}
\label{tab:image}
\begin{tabular}{|c||c|c|} \hline 
$\eta \in \IQLS(\lambda)$ & $s_{1} \xRightarrow{(\lambda, +)} \kappa(\eta)$? & $u \in W$ such that $\iota(\eta) \xRightarrow{(\lambda, -)} u$ \\ \hline 
$(e; ; 0, 1)$ & $\times$ & --- \\ 
$(s_{1}; ; 0, 1)$ & $\bigcirc$ & $s_{1}$, $e$ \\ 
$(s_{2}; ; 0, 1)$ & $\times$ & --- \\ 
$(s_{1}s_{2}; ; 0, 1)$ & $\bigcirc$ & $s_{1}s_{2}$, $s_{1}s_{2}s_{1}$ \\ 
$(s_{2}s_{1}; ; 0, 1)$ & $\bigcirc$ & $s_{2}s_{1}$, $s_{2}$ \\ 
$(s_{1}s_{2}s_{1}; ; 0, 1)$ & $\times$ & --- \\ 
$(s_{2}, e; e; 0, 1/2, 1)$ & $\times$ & --- \\ 
$(s_{1}s_{2}, s_{1}; s_{1}; 0, 1/2, 1)$ & $\bigcirc$ & $s_{1}s_{2}$, $s_{1}s_{2}s_{1}$ \\ 
$(e, s_{2}; s_{2}; 0, 1/2, 1)$ & $\times$ & --- \\ 
$(s_{1}, s_{1}s_{2}; s_{1}s_{2}; 0, 1/2, 1)$ & $\bigcirc$ & $s_{1}$, $e$ \\ 
$(s_{1}s_{2}s_{1}, s_{2}s_{1}; s_{2}s_{1}; 0, 1/2, 1)$ & $\bigcirc$ & $s_{1}s_{2}s_{1}$, $s_{1}s_{2}$ \\ 
$(s_{2}s_{1}, s_{1}s_{2}s_{1}; s_{1}s_{2}s_{1}; 0, 1/2, 1)$ & $\times$ & --- \\ \hline
\end{tabular}
\end{table}
By comparing Table~\ref{tab:image} and Table~\ref{tab:forgetful}, we can verify that 
\begin{equation}
\im(\tXi) = \left\{ (\eta, u) \in \IQLS(\lambda) \times W \ \middle| \begin{array}{ll} s_{1} \xRightarrow{(\lambda, +)} \kappa(\eta) \\ \iota(\eta) \xRightarrow{(\lambda, -)} u \end{array} \right\}.
\end{equation}
\end{example}

%%%%%%%%%%%%%%%%%
% thm:inj_forgetful %
%%%%%%%%%%%%%%%%%

\begin{theorem} \label{thm:inj_forgetful}
The forgetful map $\tXi$ is injective. 
\end{theorem}
\begin{proof}
In the proof of Theorem~\ref{thm:im_forgetful}, we constructed an admissible subset $A \in \CA(w, \Gamma_{\vtl}(\lambda))$ for each $(\eta, u) \in \CS = \im(\tXi)$. By the shellability property (Theorem~\ref{thm:shellability}), there uniquely exist the directed paths \eqref{eq:inv_forgetful1}, \eqref{eq:inv_forgetful2}, \eqref{eq:inv_forgetful3}, and \eqref{eq:inv_forgetful4} uniquely exist. 
Therefore, we obtain a (well-defined) map given by: $\im(\tXi) \rightarrow \CA(w, \Gamma_{\vtl}(\lambda)); \ (\eta, u) \mapsto A$, 
which is the inverse map of $\tXi$. Hence $\tXi$ is injective. 
This proves the theorem. 
\end{proof}

%%%%%%%%%%%%%%%%%%%%%%%%%%
% subsec:forgetful_statistics %
%%%%%%%%%%%%%%%%%%%%%%%%%%

\subsection{Statistics for admissible subsets in terms of interpolated QLS paths}\label{subsec:forgetful_statistics}
We rewrite the statistics defined for admissible subsets in terms of interpolated QLS paths. 

\begin{proposition}
For $A \in \CA(w, \Gamma_{\vtl}(\lambda))$, we have $n(A) = \nega(\Xi(A)) + \ell(\iota(\Xi(A)) \Rightarrow \ed(A))$. 
\end{proposition}
\begin{proof}
Let $A = \{ j_{1}, \ldots, j_{s} \} \in \CA(w, \Gamma_{\vtl}(\lambda))$. 
Then, $n(A)$ counts the number of roots $\gamma_{j_{k}}^{\sL}$, $k = 1, \ldots, s$, such that $\gamma_{j_{k}}^{\sL} \in \Delta^{-}$. 
Therefore, the desired equality follows since $u \xRightarrow{(\lambda, +)} \kappa(\Xi(A))$ and $\iota(\Xi(A)) \xRightarrow{(\lambda, -)} \ed(A)$. This proves the proposition. 
\end{proof}

\begin{corollary}\label{cor:forgetful_n}
For $A \in \CA(w, \Gamma_{\vtl}(\lambda))$, we have $n(A) \equiv \nega(\Xi(A)) + \ell(\ed(A)) - \ell(\iota(\Xi(A))) \mod 2$. 
\end{corollary}

Next, we rewrite $\wt(A)$ in terms of the weight of $\Xi(A)$. 

\begin{proposition}\label{prop:forgetful_wt}
For $A \in \CA(w, \Gamma_{\vtl}(\lambda))$, we have $\wt(A) = \wt(\Xi(A))$. 
\end{proposition}
\begin{proof}
We prove the proposition by induction on $|A|$. 
If $|A| = 0$, i.e., $A = \emptyset$, then by the definition of $\wt(A)$, we have $\wt(A) = w\lambda$. 
In this case, we see that $\Xi(A) = (w; ; 0, 1)$, and hence that $\wt(\Xi(A)) = w\lambda$. 
Thus we obtain $\wt(A) = \wt(\Xi(A))$. 

Let $A = \{ j_{1}, \ldots, j_{s} \} \in \CA(w, \Gamma_{\vtl}(\lambda)))$, with $s \ge 1$. 
Set $B := \{ j_{1}, \ldots, j_{s-1} \}$. 
By the induction hypothesis, we have $\wt(B) = \wt(\Xi(B))$. 

\paragraph{\underline{Step~1: $\wt(A) = \wt(B) + (-\pair{\lambda}{\overline{(\beta_{j_{s}}^{\sL})^{\vee}}} + \deg((\beta_{j_{s}}^{\sL})^{\vee})) \ed(B) \overline{\beta_{j_{s}}^{\sL}}$}.} \ \\ 

For $\nu \in \Fh_{\BR}^{\ast}$, we define $t_{\nu}: \Fh_{\BR}^{\ast} \rightarrow \Fh_{\BR}^{\ast}$ by $t_{\nu}(\xi) := \xi + \nu$. 
We have 
$s_{\gamma} t_{\nu} = t_{s_{\gamma}(\nu)} s_{\gamma}$ for $\nu \in \Fh_{\BR}^{\ast}$ and $\gamma \in \Delta^{+}$, 
and $t_{\nu_{1}}t_{\nu_{2}} = t_{\nu_{1}+\nu_{2}}$ for $\nu_{1}, \nu_{2} \in \Fh_{\BR}^{\ast}$. 
Also, for $\gamma \in \Delta$ and $k \in \BZ$, we have $t_{k\gamma} s_{|\gamma|} = s_{\gamma, k}$. 

By the definition of $\wt(A)$, we see that 
\begin{align}
\wt(A) &= -w s_{\gamma_{j_{1}}^{\sL}, -\deg((\beta_{j_{1}}^{\sL})^{\vee})} \cdots s_{\gamma_{j_{s-1}}^{\sL}, -\deg((\beta_{j_{s-1}}^{\sL})^{\vee})} s_{\gamma_{j_{s}}^{\sL}, -\deg((\beta_{j_{s}}^{\sL})^{\vee})} (-\lambda) \\ 
\begin{split}
&= -w s_{\gamma_{j_{1}}^{\sL}, -\deg((\beta_{j_{1}}^{\sL})^{\vee})} \cdots s_{\gamma_{j_{s-1}}^{\sL}, -\deg((\beta_{j_{s-1}}^{\sL})^{\vee})} (-\lambda \\ 
& \hspace{60mm} -(\pair{-\lambda}{(\gamma_{j_{s}}^{\sL})^{\vee}} + \deg((\beta_{j_{s}}^{\sL})^{\vee})) \gamma_{j_{s}}^{\sL})
\end{split} \\ 
\begin{split}
&= -w s_{\gamma_{j_{1}}^{\sL}, -\deg((\beta_{j_{1}}^{\sL})^{\vee})} \cdots s_{\gamma_{j_{s-1}}^{\sL}, -\deg((\beta_{j_{s-1}}^{\sL})^{\vee})} (-\lambda \\ 
& \hspace{60mm} -(-\pair{\lambda}{\overline{(\beta_{j_{s}}^{\sL})^{\vee}}} + \deg((\beta_{j_{s}}^{\sL})^{\vee})) \overline{\beta_{j_{s}}^{\sL}}). 
\end{split}  
\end{align}
We set $d := -\pair{\lambda}{\overline{(\beta_{j_{s}}^{\sL})^{\vee}}} + \deg((\beta_{j_{s}}^{\sL})^{\vee})$. 
By repeated application of the identity $t_{k\gamma} s_{|\gamma|} = s_{\gamma, k}$ for $\gamma \in \Delta$ and $k \in \BZ$, 
we deduce that there exists $\xi \in \Fh_{\BR}^{\ast}$ (in fact, $\xi \in Q$) such that 
\begin{equation}
w s_{\gamma_{j_{1}}^{\sL}, -\deg((\beta_{j_{1}}^{\sL})^{\vee})} \cdots s_{\gamma_{j_{s-1}}^{\sL}, -\deg((\beta_{j_{s-1}}^{\sL})^{\vee})} = w s_{|\gamma_{j_{1}}^{\sL}|} \cdots s_{|\gamma_{j_{s-1}}^{\sL}|} t_{\xi}. 
\end{equation}
Therefore, we compute: 
\begin{align}
\wt(A) &= -w s_{|\gamma_{j_{1}}^{\sL}|} \cdots s_{|\gamma_{j_{s-1}}^{\sL}|} t_{\xi} (-\lambda - d\overline{\beta_{j_{s}}^{\sL}}) \\ 
&= -w s_{|\gamma_{j_{1}}^{\sL}|} \cdots s_{|\gamma_{j_{s-1}}^{\sL}|} (-\lambda - d\overline{\beta_{j_{s}}^{\sL}} + \xi) \\ 
&= -w s_{|\gamma_{j_{1}}^{\sL}|} \cdots s_{|\gamma_{j_{s-1}}^{\sL}|} (-\lambda + \xi) + d w s_{|\gamma_{j_{1}}^{\sL}|} \cdots s_{|\gamma_{j_{s-1}}^{\sL}|} (\overline{\beta_{j_{s}}^{\sL}}) \\ 
&= -w s_{|\gamma_{j_{1}}^{\sL}|} \cdots s_{|\gamma_{j_{s-1}}^{\sL}|}t_{\xi} (-\lambda) + d w s_{|\gamma_{j_{1}}^{\sL}|} \cdots s_{|\gamma_{j_{s-1}}^{\sL}|} (\overline{\beta_{j_{s}}^{\sL}}) \\ 
\begin{split}
&= \underbrace{-w s_{\gamma_{j_{1}}^{\sL}, -\deg((\beta_{j_{1}}^{\sL})^{\vee})} \cdots s_{\gamma_{j_{s-1}}^{\sL}, -\deg((\beta_{j_{s-1}}^{\sL})^{\vee})}(-\lambda)}_{\wt(B)} \\ 
& \hspace{50mm} + d \underbrace{w s_{|\gamma_{j_{1}}^{\sL}|} \cdots s_{|\gamma_{j_{s-1}}^{\sL}|}}_{\ed(B)} (\overline{\beta_{j_{s}}^{\sL}}) 
\end{split} \\ 
& = \wt(B) + (-\pair{\lambda}{\overline{(\beta_{j_{s}}^{\sL})^{\vee}}} + \deg((\beta_{j_{s}}^{\sL})^{\vee})) \ed(B) (\overline{\beta_{j_{s}}^{\sL}}), 
\end{align}
as desired. 

\paragraph{\underline{Step~2: $\wt(\Xi(A)) = \wt(\Xi(B)) + (-\pair{\lambda}{\overline{(\beta_{j_{s}}^{\sL})^{\vee}}} + \deg((\beta_{j_{s}}^{\sL})^{\vee})) \ed(B) \overline{\beta_{j_{s}}^{\sL}}$.}} \ \\ 

We have the following three cases: 
\begin{enu}
\item $d_{j_{s-1}} < d_{j_{s}} < 1$; 
\item $d_{j_{s-1}} = d_{j_{s}} < 1$; 
\item $d_{j_{s}} = 1$. 
\end{enu}

Assume case (1). Then we have $m_{t-1} = s-1$ and $m_{t} = s$. 
Hence 
\begin{equation}
\Xi(B) = (u_{m_{t-1}}, \ldots, u_{m_{1}}; u_{n_{t-2}}, \ldots, u_{n_{1}}; 0, 1-d_{j_{m_{t-1}}}, \ldots, 1-d_{j_{m_{2}}}, 1), 
\end{equation}
and 
\begin{equation}
\Xi(A) = (u_{m_{t}}, \ldots, u_{m_{1}}; u_{n_{t-1}}, \ldots, u_{n_{1}}; 0, 1-d_{j_{m_{t}}}, \ldots, 1-d_{j_{m_{2}}}, 1). 
\end{equation}
Note that $u_{m_{t}} = u_{m_{t-1}} s_{|\gamma_{j_{s}}^{\sL}|}$. 
Also, note that $\ed(B) = u_{m_{t-1}}$. 
Therefore, by the definition of weights, we compute: 
\begin{align}
& \wt(\Xi(A)) \\
\begin{split}
&= \wt(\Xi(B)) - (1-d_{j_{m_{t-1}}}) u_{m_{t-1}} \lambda \\ 
& \hspace{30mm} + ((1-d_{j_{m_{t-1}}}) - (1-d_{j_{m_{t}}})) u_{m_{t-1}} \lambda + (1-d_{j_{m_{t}}}) u_{m_{t}} \lambda 
\end{split} \\
&= \wt(\Xi(B)) - (1-d_{j_{m_{t}}}) u_{m_{t-1}} \lambda + (1-d_{j_{m_{t}}}) u_{m_{t}} \lambda \\
&= \wt(\Xi(B)) - (1-d_{j_{m_{t}}}) u_{m_{t-1}} \lambda + (1-d_{j_{m_{t}}}) u_{m_{t-1}} s_{|\gamma_{j_{s}}^{\sL}|} \lambda \\
&= \wt(\Xi(B)) - (1-d_{j_{m_{t}}}) u_{m_{t-1}} \lambda + (1-d_{j_{m_{t}}}) u_{m_{t-1}} (\lambda - \pair{\lambda}{|\gamma_{j_{s}}^{\sL}|^{\vee}} |\gamma_{j_{s}}^{\sL}|) \\
&= \wt(\Xi(B)) - (1-d_{j_{m_{t}}}) \pair{\lambda}{|\gamma_{j_{s}}^{\sL}|^{\vee}} u_{m_{t-1}} |\gamma_{j_{s}}^{\sL}| \\
&= \wt(\Xi(B)) - (1-d_{j_{m_{t}}}) \pair{\lambda}{(\gamma_{j_{s}}^{\sL})^{\vee}} u_{m_{t-1}} \gamma_{j_{s}}^{\sL} \\
&= \wt(\Xi(B)) - \left( 1-\frac{\deg((\beta_{j_{s}}^{\sL})^{\vee})}{\pair{\lambda}{\overline{(\beta_{j_{s}}^{\sL})^{\vee}}}} \right) \pair{\lambda}{\overline{(\beta_{j_{s}}^{\sL})^{\vee}}} \ed(B) \overline{\beta_{j_{s}}^{\sL}} \\
&= \wt(\Xi(B)) + (-\pair{\lambda}{\overline{(\beta_{j_{s}}^{\sL})^{\vee}}} + \deg((\beta_{j_{s}}^{\sL})^{\vee})) \ed(B) \overline{\beta_{j_{s}}^{\sL}}, 
\end{align}
as desired. 

Next, assume case (2). Observe that 
\begin{equation}
\Xi(B) = (u_{m_{t}-1}, u_{m_{t-1}}, \ldots, u_{m_{1}}; z, u_{n_{t-2}}, \ldots, u_{n_{1}}; 0, 1-d_{j_{m_{t}}}, \ldots, 1-d_{j_{m_{2}}}, 1), 
\end{equation}
where 
\begin{equation}
z = \begin{cases} u_{n_{t-1}}, & \text{if $u_{n_{t-1}} \not= u_{m_{t}}$}, \\ u_{m_{t}-1}, & \text{if $u_{n_{t-1}} = u_{m_{t}}$}, \end{cases}
\end{equation}
and that 
\begin{equation}
\Xi(A) = (u_{m_{t}}, \ldots, u_{m_{1}}; u_{n_{t-1}}, \ldots, u_{n_{1}}; 0, 1-d_{j_{m_{t}}}, \ldots, 1-d_{j_{m_{2}}}, 1). 
\end{equation}
Note that $m_{t} = s$, $m_{t-1} < s-1$, and that $u_{m_{t}} = u_{m_{t}-1} s_{|\gamma_{j_{s}}^{\sL}|}$. 
Also, note that $\ed(B) = u_{m_{t}-1}$. 
Therefore, we compute: 
\begin{align}
\wt(\Xi(A)) &= \wt(\Xi(B)) - (1-d_{j_{m_{t}}}) u_{m_{t}-1} \lambda + (1-d_{j_{m_{t}}}) u_{m_{t}} \lambda \\
&= \wt(\Xi(B)) - (1-d_{j_{m_{t}}}) u_{m_{t}-1} \lambda + (1-d_{j_{m_{t}}}) u_{m_{t}-1} s_{|\gamma_{j_{s}}^{\sL}|} \lambda \\
\begin{split}
&= \wt(\Xi(B)) - (1-d_{j_{m_{t}}}) u_{m_{t}-1} \lambda \\ 
& \hspace{30mm} + (1-d_{j_{m_{t}}}) u_{m_{t}-1} (\lambda - \pair{\lambda}{|\gamma_{j_{s}}^{\sL}|^{\vee}} |\gamma_{j_{s}}^{\sL}|) 
\end{split} \\
&= \wt(\Xi(B)) - (1-d_{j_{m_{t}}}) \pair{\lambda}{|\gamma_{j_{s}}^{\sL}|^{\vee}} u_{m_{t}-1} |\gamma_{j_{s}}^{\sL}| \\
&= \wt(\Xi(B)) - (1-d_{j_{m_{t}}}) \pair{\lambda}{(\gamma_{j_{s}}^{\sL})^{\vee}} u_{m_{t}-1} \gamma_{j_{s}}^{\sL} \\
&= \wt(\Xi(B)) - \left( 1- \frac{\deg((\beta_{j_{s}}^{\sL})^{\vee})}{\pair{\lambda}{\overline{(\beta_{j_{s}}^{\sL})^{\vee}}}} \right) \pair{\lambda}{\overline{(\beta_{j_{s}}^{\sL})^{\vee}}} \ed(B) \overline{\beta_{j_{s}}^{\sL}} \\
&= \wt(\Xi(B)) + (-\pair{\lambda}{\overline{(\beta_{j_{s}}^{\sL})^{\vee}}} + \deg((\beta_{j_{s}}^{\sL})^{\vee})) \ed(B) \overline{\beta_{j_{s}}^{\sL}}, 
\end{align}
as desired.

Finally, in case (3), we see that $\Xi(A) = \Xi(B)$. Hence $\wt(\Xi(A)) = \wt(\Xi(B))$. 
Since $d_{j_{s}} = \deg((\beta_{j_{s}}^{\sL})^{\vee})/\pair{\lambda}{\overline{(\beta_{j_{s}}^{\sL})^{\vee}}} = 1$, 
we have $-\pair{\lambda}{\overline{(\beta_{j_{s}}^{\sL})^{\vee}}} + \deg((\beta_{j_{s}}^{\sL})^{\vee}) = 0$. 
This proves the desired equality. 

\paragraph{\underline{Step~3: $\wt(A) = \wt(\Xi(A))$}} \ \\ 

By Steps~1 and 2, we deduce that 
\begin{align}
\wt(A) &= \wt(B) + (-\pair{\lambda}{\overline{(\beta_{j_{s}}^{\sL})^{\vee}}} + \deg((\beta_{j_{s}}^{\sL})^{\vee})) \ed(B) \overline{\beta_{j_{s}}^{\sL}} \\
&= \wt(\Xi(B)) + (-\pair{\lambda}{\overline{(\beta_{j_{s}}^{\sL})^{\vee}}} + \deg((\beta_{j_{s}}^{\sL})^{\vee})) \ed(B) \overline{\beta_{j_{s}}^{\sL}} \\
&= \wt(\Xi(A)), 
\end{align}
where we have used the induction hypothesis for the 2nd equality. 
This completes the proof of the proposition. 
\end{proof}

\begin{example}
Assume that $\Fg$ is of type $A_{2}$. Let $\lambda = -\vpi_{1}+2\vpi_{2}$ and $w = s_{1}$. 
Table~\ref{tab:admissible} gives the list of $\wt(A)$ for all $A \in \CA(s_{1}, \Gamma_{\vtl}(\lambda))$. 
Also, we computed $\wt(\eta)$ for all $\eta \in \IQLS(\lambda)$ in Example~\ref{ex:IQLS_n=2}. 
Hence, by comparing Table~\ref{tab:admissible} and Table~\ref{tab:forgetful}, we can verify that $\wt(A) = \wt(\Xi(A))$ for all $A \in \CA(s_{1}, \Gamma_{\vtl}(\lambda))$. 

\end{example}

\subsection{The ``\texorpdfstring{$q = 0$}{q=0}'' counterpart of the forgetful map}
If $A \in \CA|_{q=0}(w, \Gamma_{\vtl}(\lambda))$, then $\bp(A)$ is a directed path in $\BG(W)$. 
Therefore, by the definition of $\ILS(\lambda)$, we obtain the following. 

%%%%%%%%%%%%%%%%%
% cor:forgetful_q=0 %
%%%%%%%%%%%%%%%%%

\begin{corollary} \label{cor:forgetful_q=0}
We have $\Xi(\CA|_{q=0}(w, \Gamma_{\vtl}(\lambda))) \subset \ILS(\lambda)$. 
Moreover, we have 
%
%%%%%%%%%%%%%%%%%%%%%%
% eq:image_forgetful_q=0 %
%%%%%%%%%%%%%%%%%%%%%%
%
\begin{equation}
\tXi(\CA|_{q=0}(w, \Gamma_{\vtl}(\lambda))) = \left\{ (\eta, u) \in \ILS(\lambda) \times W \ \middle| \begin{array}{ll} w \xRightarrow{(\lambda, +, q=0)} \kappa(\eta) \\ \iota(\eta) \xRightarrow{(\lambda, -, q=0)} u \end{array} \right\}. \label{eq:image_forgetful_q=0} 
\end{equation}
\end{corollary}

%%%%%%%%%%%%%%%%
% sec:equivariant %
%%%%%%%%%%%%%%%%

\section{Application}\label{sec:equivariant}
As an application of the forgetful map, 
we rewrite an identity of Chevalley type for the graded characters of level-zero Demazure submodules, given in \cite{KLN}. 
In this section, we take and fix an arbitrary $\lambda \in P$. Let us take a reflection order ${\vtl} \in \RO(\lambda, \Delta^{+})$, and the suitable $\lambda$-chain $\Gamma_{\vtl}(\lambda)$ for $\vtl$. 

Let $U_{\q}(\Fg_{\af})$ denote the quantum affine algebra associated to $\Fg_{\af}$ with Chevalley generators $E_{i}, F_{i} \in U_{\q}(\Fg_{\af})$, $i \in I_{\af} = I \sqcup \{0\}$, 
where $\q$ is an indeterminate. 
We denote by $U_{\q}^{-}(\Fg_{\af}) := \langle F_{i} \rangle_{i \in I_{\af}} \subset U_{\q}(\Fg_{\af})$ 
the subalgebra of $U_{\q}(\Fg_{\af})$ generated by $\{F_{i} \mid i \in I_{\af}\}$. 

For each $\lambda \in P^{+}$, we denote by $V(\lambda)$ 
the \emph{level-zero extremal weight module} of extremal weight $\lambda$ over $U_{\q}(\Fg_{\af})$, 
which is equipped with a family $\{v_{x}\}_{x \in W_{\af}} \subset V(\lambda)$ of extremal weight vectors, 
where $v_{x} \in V(\lambda)$, $x \in W_{\af}$, is an extremal weight vector of weight $x\lambda$ (see \cite[Proposition~8.2.2]{Kas}). 
For $x \in W_{\af}$ and $\lambda \in P^{+}$, the \emph{Demazure submodule} $V_{x}^{-}(\lambda)$ of $V(\lambda)$ is defined by $V_{x}^{-}(\lambda) := U_{\q}^{-}(\Fg_{\af})v_{x}$. 
The module $V_{x}^{-}(\lambda)$ has a weight space decomposition with respect to the Cartan subalgebra $\Fh_{\af}$: 
\begin{equation}
V_{x}^{-}(\lambda) = \bigoplus_{\gamma \in Q, k \in \BZ} V_{x}^{-}(\lambda)_{\lambda + \gamma + k\delta}
\end{equation}
with finite-demensional (over $\BC(\q)$) weight spaces $V_{x}^{-}(\lambda)_{\lambda + \gamma + k\delta}$, $\gamma \in Q$ and $k \in \BZ$. 
As in \cite[(2.21)]{KNS}, we define the \emph{graded character} $\gch V_{x}^{-}(\lambda)$ of $V_{x}^{-}(\lambda)$ by 
\begin{equation}
\gch V_{x}^{-}(\lambda) := \sum_{\gamma \in Q, k \in \BZ} \dim(V_{x}^{-}(\lambda)_{\lambda + \gamma + k\delta}) q^{k} e^{\lambda + \gamma} \in \BZ[P]\pra{q^{-1}}, 
\end{equation}
where $q$ is an indeterminate (not to be confused with $\q$). 
If $x = w t_{\xi}$ with $w \in W$ and $\xi \in Q^{\vee}$, then we know that $\gch V_{x}^{-}(\lambda) \in \mathbb{Z}[P]\bra{q^{-1}}q^{-\pair{\lambda}{\xi}}$ (see \cite[(2.22)]{KNS}); 
in fact, we have $\gch V_{w}^{-}(\lambda) \in \mathbb{Z}\bra{q^{-1}}[P]$ for $w \in W$. 

To describe the identity of Chevalley type for graded characters, we need some notation for partitions. 
Let $\lambda \in P$, and write it as $\lambda = \sum_{i \in I} m_{i}\vpi_{i}$. 
We define the set $\bPar(\lambda)$ by 
%
%%%%%%%%%%%%
% eq:def_Par %
%%%%%%%%%%%%
%
\begin{equation}\label{eq:def_Par}
\bPar(\lambda) := \left\{ \bchi = (\chi^{(i)})_{i \in I} \ \middle| \ \parbox{15em}{$\chi^{(i)}$ is a partition whose length is less than or equal to $\max\{m_{i}, 0\}$} \right\}. 
\end{equation}
For $\bchi = (\chi^{(i)})_{i \in I} \in \bPar(\lambda)$, we write it as $\chi^{(i)} = (\chi_{1}^{(i)} \ge \chi_{2}^{(i)} \ge \cdots \ge \chi_{l_{i}}^{(i)} > 0)$, 
where $0 \le l_{i} \le \max\{m_{i}, 0\}$ and $\chi_{1}^{(i)}, \ldots, \chi_{l_{i}}^{(i)} \in \BZ$, and set 
%
%%%%%%%%%%%%%%%%%%%%%
% eq:def_statistics_par %
%%%%%%%%%%%%%%%%%%%%%
%
\begin{equation}\label{eq:def_statistics_par}
|\bchi| := \sum_{i \in I} \sum_{k = 1}^{l_{i}} \chi_{k}^{(i)}, \quad \iota(\bchi) := \sum_{i \in I} \chi_{1}^{(i)} \alpha_{i}^{\vee}; 
\end{equation}
if $\chi^{(i)} = \emptyset$, then we understand that $l_{i} = 0$ and $\chi_{1}^{(i)} = 0$. 

The following is the \emph{identity of Chevalley type} for graded characters, given in \cite{KLN}; we should mention that it is a ``representation-theoretic'' analog of the Chevalley formula for the equivariant $K$-group of semi-infinite flag manifolds, given in \cite{LNS}. 
%
%%%%%%%%%%%%%%%%%%%%%
% thm:PC-type_formula %
%%%%%%%%%%%%%%%%%%%%%
%
\begin{theorem}[{\cite[Theorem~5.16]{KLN}}]\label{thm:PC-type_formula}
Let $x \in W_{\af}$, and write it 
as $x = w t_{\xi}$, with $w \in W$ and $\xi \in Q^{\vee}$. 
Let $\mu \in P^{+}$ and $\lambda \in P$ be such that $\mu + \lambda \in P^{+}$, 
and let $\Gamma$ be an arbitrary reduced $\lambda$-chain. 
Then we have 
\begin{equation}
\begin{split}\label{eq:PC-type_formula}
& \gch V_{x}^{-}(\mu + \lambda) = \\ 
& \sum_{A \in \CA(w, \Gamma)} \sum_{\bchi \in \bPar(\lambda)} (-1)^{n(A)} q^{-\height(A) - \pair{\lambda}{\xi} - |\bchi|} e^{\wt(A)} \gch V_{\ed(A)t_{\xi + \down(A) + \iota(\bchi)}}^{-}(\mu). 
\end{split}
\end{equation}
\end{theorem}

\begin{remark}\label{rhs-cancel}
The right-hand side of \eqref{eq:PC-type_formula} is zero if $\mu + \lambda \notin P^{+}$; 
see \cite[Appendix~B]{KLN}. 
\end{remark}

The aim of this section is to rewrite the identity above in terms of interpolated QLS paths. 
We introduce some additional statistics for interpolated QLS paths; 
in particular, the degree function $\Deg_{w}$ is a generalization of that for QLS paths (see \cite[Section~4]{LNSSS2}). 

\begin{definition}
Let $\eta = (x_{1}, \ldots, x_{s}; y_{1}, \ldots, y_{s-1}; \sigma_{0}, \ldots, \sigma_{s}) \in \IQLS(\lambda)$. Let $w, u \in W$ 
be such that $w \xRightarrow{(\lambda, +)} \kappa(\eta)$ and $\iota(\eta) \xRightarrow{(\lambda, -)} u$. 
We define $\xi(u, \eta, w)$ by 
\begin{equation}
\begin{split}
& \xi(u, \eta, w) := \\ 
& \wt(w \Rightarrow \kappa(\eta)) + \sum_{k = 1}^{s-1} (\wt(x_{k+1} \Rightarrow y_{k}) + \wt(y_{k} \Rightarrow x_{k})) + \wt(\iota(\eta) \Rightarrow u). 
\end{split}
\end{equation}
\end{definition}

\begin{definition}
Let $\eta = (x_{1}, \ldots, x_{s}; y_{1}, \ldots, y_{s-1}; \sigma_{0}, \ldots, \sigma_{s}) \in \IQLS(\lambda)$, and let $w \in W$ 
be such that $w \xRightarrow{(\lambda, +)} \kappa(\eta)$. We define $\Deg_{w}(\eta)$ by 
\begin{equation}
\Deg_{w}(\eta) := - \pair{\lambda}{\wt(w \Rightarrow \kappa(\eta))} - \sum_{k = 1}^{s-1} \sigma_{k} \pair{\lambda}{\wt(x_{k+1} \Rightarrow y_{k}) + \wt(y_{k} \Rightarrow x_{k})}. 
\end{equation}
\end{definition}

\begin{example}
Assume that $\Fg$ is of type $A_{2}$. Let $\lambda = -\vpi_{1} + 2\vpi_{2}$, $w = s_{1}$. 
Let $\eta = (s_{1}, s_{1}s_{2}; s_{1}s_{2}; 0, 1/2, 1) \in \IQLS(\lambda)$. 
We see that $s_{1} \xRightarrow{(\lambda, +)} \kappa(\eta) = s_{1}s_{2}$ and $\iota(\eta) = s_{1} \xRightarrow{(\lambda, -)} e$. 
Let us compute $\xi(e, \eta, s_{1})$ and $\Deg_{s_{1}}(\eta)$. 

First, observe that the directed path $s_{1} \Rightarrow \kappa(\eta) = s_{1}s_{2}$ is just the edge $s_{1} \xrightarrow{\alpha_{2}} s_{1}s_{2}$. 
Since this is a Bruhat edge, we have $\wt(s_{1} \Rightarrow \kappa(\eta)) = 0$. 
Also, since the directed path $s_{1}s_{2} \Rightarrow s_{1}s_{2}$ is the trivial one, we have $\wt(s_{1}s_{2} \Rightarrow s_{1}s_{2}) = 0$. 
Next, the directed path $s_{1}s_{2} \Rightarrow s_{1} = \iota(\eta)$ is just the quantum edge $s_{1}s_{2} \xrightarrow{\alpha_{2}} s_{1}$, and hence 
$\wt(s_{1}s_{2} \Rightarrow s_{2}) = \alpha_{2}^{\vee}$. 
In addition, the directed path $\iota(\eta) = s_{1} \Rightarrow e$ is just the quantum edge $s_{1} \xrightarrow{\alpha_{1}} e$, and hence 
$\wt(\iota(\eta) \Rightarrow e) = \alpha_{1}^{\vee}$. 
From these, we see that 
\begin{align}
\xi(e, \eta, s_{1}) &= \wt(s_{1} \Rightarrow \kappa(\eta)) + (\wt(s_{1}s_{2} \Rightarrow s_{1}s_{2}) + \wt(s_{1}s_{2} \Rightarrow s_{1})) + \wt(\iota(\eta) \Rightarrow e) \\ 
&= \alpha_{1}^{\vee} + \alpha_{2}^{\vee}, \\ 
\Deg_{w}(\eta) &= -\pair{\lambda}{\wt(s_{1} \Rightarrow \kappa(\eta))} - \frac{1}{2} \pair{\lambda}{\wt(s_{1}s_{2} \Rightarrow s_{1}s_{2}) + \wt(s_{1}s_{2} \Rightarrow s_{1})} \\ 
&= -\frac{1}{2} \pair{-\vpi_{1}+2\vpi_{2}}{\alpha_{2}^{\vee}} \\ 
&= -1. 
\end{align}

\end{example}

We will describe the relation between the statistics above for interpolated QLS paths and those for admissible subsets. 
The following lemma is obvious by the definitions of $\down(A)$ and $\xi(u, \eta, w)$. 

%%%%%%%%%%%%%%%%
% lem:down_to_xi %
%%%%%%%%%%%%%%%%

\begin{lemma} \label{lem:down_to_xi}
For $A \in \CA(w, \Gamma_{\vtl}(\lambda))$, we write $\tXi(A) = (\eta, u)$. 
Then, we have $\down(A) = \xi(u, \eta, w)$. 
\end{lemma}

Next, we consider the relation between $\height(A)$ and $\Deg_{w}(\eta)$. 

%%%%%%%%%%%%%%%%%%
% lem:height_to_deg %
%%%%%%%%%%%%%%%%%%

\begin{lemma} \label{lem:height_to_deg}
Let $w \in W$. 
For $A \in \CA(w, \Gamma_{\vtl}(\lambda))$, we write $\Xi(A) = \eta$. 
Then, we have $\height(A) = - \Deg_{w}(\eta)$. 
\end{lemma}
\begin{proof}
We use the notation of Section~\ref{sec:forgetful}. 
In particular, recall that 
\begin{equation}
\begin{split}
\Xi(A) &= (u_{m_{t}}, u_{m_{t-1}}, \ldots, u_{m_{1}}; u_{n_{t-1}}, u_{n_{t-2}}, \ldots, u_{n_{1}}; \\ 
& \hspace{8mm} 0, 1-d_{j_{m_{t}}}, 1-d_{j_{m_{t-1}}}, \ldots, 1-d_{j_{m_{2}}}, 1). 
\end{split}
\end{equation}
Then, by the definition of $\Deg_{w}(\eta)$, we have 
\begin{equation}
\begin{split}
\Deg_{w}(\eta) &= - \pair{\lambda}{\wt(w \Rightarrow u_{m_{1}})} \\
& \hspace{5mm} - \sum_{k = 1}^{t-1} (1-d_{j_{m_{k+1}}}) \pair{\lambda}{\wt(u_{m_{k}} \Rightarrow u_{n_{k}}) + \wt(u_{n_{k}} \Rightarrow u_{m_{k+1}})}. 
\end{split}
\end{equation}
Let $k = 1, \ldots, t-1$, and consider the following directed paths in $\QBG(W)$: 
\begin{equation}
u_{m_{k}} \xrightarrow{|\gamma_{j_{m_{k}+1}}^{\sL}|} u_{m_{k}+1} \xrightarrow{|\gamma_{j_{m_{k}+2}}^{\sL}|} \cdots \xrightarrow{|\gamma_{j_{n_{k}}}^{\sL}|} u_{n_{k}}, \label{eq:path1} %eq:path1
\end{equation}
\begin{equation}
u_{n_{k}} \xrightarrow{|\gamma_{j_{n_{k}+1}}^{\sL}|} u_{n_{k}+1} \xrightarrow{|\gamma_{j_{n_{k}+2}}^{\sL}|} \cdots \xrightarrow{|\gamma_{j_{m_{k+1}}}^{\sL}|} u_{m_{k+1}}. \label{eq:path2} %eq:path2
\end{equation}
It follows that 
\begin{align}
\wt(u_{m_{k}} \Rightarrow u_{n_{k}}) &= \sum_{\substack{m_{k}+1 \le p \le n_{k} \\ u_{p-1} \rightarrow u_{p} \text{ is a quantum edge}}} |\gamma_{j_{p}}^{\sL}|^{\vee}, \\
\wt(u_{n_{k}} \Rightarrow u_{m_{k+1}}) &= \sum_{\substack{n_{k}+1 \le p \le m_{k+1} \\ u_{p-1} \rightarrow u_{p} \text{ is a quantum edge}}} |\gamma_{j_{p}}^{\sL}|^{\vee}. 
\end{align}
For $m_{k}+1 \le p \le m_{k+1}$, we see that 
\begin{align}
(1 - d_{j_{m_{k+1}}}) \pair{\lambda}{|\gamma_{j_{p}}^{\sL}|^{\vee}} &= \sgn(\gamma_{j_{p}}^{\sL}) (1 - d_{j_{m_{k+1}}}) \pair{\lambda}{(\gamma_{j_{p}}^{\sL})^{\vee}} \\ 
&= \sgn(\gamma_{j_{p}}^{\sL}) \left( 1 - \frac{\deg((\beta_{j_{p}}^{\sL})^{\vee})}{\pair{\lambda}{\overline{(\beta_{j_{p}}^{\sL})^{\vee}}}} \right) \pair{\lambda}{\overline{(\beta_{j_{p}}^{\sL})^{\vee}}} \\ 
&= \sgn(\gamma_{j_{p}}^{\sL}) (\pair{\lambda}{\overline{(\beta_{j_{p}}^{\sL})^{\vee}}} - \deg((\beta_{j_{p}}^{\sL})^{\vee})). 
\end{align}
Hence we deduce that 
\begin{align}
& \sum_{k=1}^{t-1} (1 - d_{j_{m_{k+1}}}) \pair{\lambda}{\wt(u_{m_{k}} \Rightarrow u_{n_{k}}) + \wt(u_{n_{k}} \Rightarrow u_{m_{k+1}})} \\ 
&= \sum_{\substack{m_{1}+1 \le p \le m_{t} \\ u_{p-1} \rightarrow u_{p} \text{ is a quantum edge}}} (1 - d_{j_{m_{k+1}}}) \pair{\lambda}{|\gamma_{j_{p}}^{\sL}|^{\vee}} \\ 
&= \sum_{\substack{m_{1}+1 \le p \le m_{t} \\ u_{p-1} \rightarrow u_{p} \text{ is a quantum edge}}} \sgn(\gamma_{j_{p}}^{\sL}) (\pair{\lambda}{(\gamma_{j_{p}}^{\sL})^{\vee}} - \deg((\beta_{j_{p}}^{\sL})^{\vee})). 
\end{align}

Also, consider the following directed path in $\QBG(W)$:
\begin{equation}
w \xrightarrow{|\gamma_{j_{1}}^{\sL}|} u_{1} \xrightarrow{|\gamma_{j_{2}}^{\sL}|} \cdots \xrightarrow{|\gamma_{j_{m_{1}}}^{\sL}|} u_{m_{1}}. 
\end{equation}
Since $w \xRightarrow{(\lambda, +)} u_{m_{1}} = \kappa(\eta)$, 
it follows that $\sgn(\gamma_{j_{p}}^{\sL}) = 1$ and $\pair{\lambda}{|\gamma_{j_{p}}^{\sL}|^{\vee}} = \pair{\lambda}{(\gamma_{j_{p}}^{\sL})^{\vee}}$ for $1 \le p \le m_{1}$. 
Also, we see that $\deg((\beta_{j_{p}}^{\sL})^{\vee}) = 0$ for $1 \le p \le m_{1}$. Hence we deduce that 
\begin{align}
& \pair{\lambda}{\wt(w \Rightarrow \kappa(\eta))} \\ 
&= \sum_{\substack{1 \le p \le m_{1} \\ u_{p-1} \rightarrow u_{p} \text{ is a quantum edge}}} \pair{\lambda}{|\gamma_{j_{p}}^{\sL}|^{\vee}}\\ 
&= \sum_{\substack{1 \le p \le m_{1} \\ u_{p-1} \rightarrow u_{p} \text{ is a quantum edge}}} \sgn(\gamma_{j_{p}}^{\sL}) (\pair{\lambda}{(\gamma_{j_{p}}^{\sL})^{\vee}} - \deg((\beta_{j_{p}}^{\sL})^{\vee})). 
\end{align}

Finally, observe that $\pair{\lambda}{(\gamma_{j_{p}}^{\sL})^{\vee}} - \deg((\beta_{j_{p}}^{\sL})^{\vee}) = 0$ for $m_{t}+1 \le p \le s$. 
Therefore, we conclude that 
\begin{align}
\Deg_{w}(\eta) &= -\sum_{\substack{1 \le p \le m_{t} \\ u_{p-1} \rightarrow u_{p} \text{ is a quantum edge}}} \sgn(\gamma_{j_{p}}^{\sL}) (\pair{\lambda}{(\gamma_{j_{p}}^{\sL})^{\vee}} - \deg((\beta_{j_{p}}^{\sL})^{\vee})) \\ 
&= -\sum_{\substack{1 \le p \le s \\ u_{p-1} \rightarrow u_{p} \text{ is a quantum edge}}} \sgn(\gamma_{j_{p}}^{\sL}) (\pair{\lambda}{(\gamma_{j_{p}}^{\sL})^{\vee}} - \deg((\beta_{j_{p}}^{\sL})^{\vee})) \\ 
&= -\height(A), 
\end{align}
as desired. This proves the lemma. 
\end{proof}

\begin{example}
Assume that $\Fg$ is of type $A_{2}$. Let $\lambda = -\vpi_{1}+2\vpi_{2}$ and $w = s_{1}$. In Table~\ref{tab:xi_deg} below, we have given the list of $\xi(\ed(A), \Xi(A), s_{1})$ and $\Deg_{s_{1}}(\Xi(A))$ for all $A \in \CA(s_{1}, \Gamma_{\vtl}(\lambda))$. 
\begin{table}[ht]
\centering
\caption{The list of $\xi(\ed(A), \Xi(A), s_{1})$ and $\Deg_{s_{1}}(\Xi(A))$ for all $A \in \CA(s_{1}, \Gamma_{\vtl}(\lambda)$}
\label{tab:xi_deg}
\begin{tabular}{|c||cc|} \hline
$A$ & $\xi(\ed(A), \Xi(A), s_{1})$ & $\Deg_{s_{1}}(\Xi(A))$ \\ \hline 
$\emptyset$ & 0 & $0$ \\ 
$\{1\}$ & $0$ & $0$ \\ 
$\{2\}$ & $0$ & $0$ \\ 
$\{3\}$ & $0$ & $0$ \\
$\{4\}$ & $\alpha_{1}^{\vee}$ & $0$ \\ 
$\{1, 3\}$ & $\alpha_{2}^{\vee}$ & $-1$ \\ 
$\{1, 4\}$ & $0$ & $0$ \\ 
$\{2, 3\}$ & $0$ & $0$ \\ 
$\{2, 4\}$ & $\alpha_{1}^{\vee}$ & $0$ \\ 
$\{3, 4\}$ & $0$ & $0$ \\ 
$\{1, 3, 4\}$ & $\alpha_{1}^{\vee} + \alpha_{2}^{\vee}$ & $-1$ \\ 
$\{2, 3, 4\}$ & $\alpha_{1}^{\vee}$ & $0$ \\ \hline
\end{tabular}
\end{table}

By comparing Table~\ref{tab:xi_deg} and Table~\ref{tab:admissible}, we see that $\xi(\ed(A), \Xi(A), s_{1}) = \down(A)$ and $\Deg_{s_{1}}(\Xi(A)) = -\height(A)$ for all $A \in \CA(s_{1}, \Gamma_{\vtl}(\lambda))$. 
\end{example}

Finally, we obtain the following description of the identity of Chevalley type for graded characters 
in terms of interpolated QLS paths. 

\begin{theorem} \label{thm:gch_Chevalley}
Let $x \in W_{\af}$, and write it 
as $x = w t_{\xi}$, with $w \in W$ and $\xi \in Q^{\vee}$. 
Let $\mu \in P^{+}$ and $\lambda \in P$ be such that $\mu + \lambda \in P^{+}$. 
Then, the follwing equality holds: 
\begin{equation}
\begin{split}
& \gch V_{x}^{-}(\mu + \lambda) =\\ 
& \sum_{\substack{\eta \in \IQLS(\lambda) \\ w \xRightarrow{(\lambda, +)} \kappa(\eta)}} \sum_{\substack{u \in W \\ \iota(\eta) \xRightarrow{(\lambda, -)} u}} \sum_{\bchi \in \bPar(\lambda)} (-1)^{\nega(\eta) + \ell(u) - \ell(\iota(\eta))} q^{\Deg_{w}(\eta) - \pair{\lambda}{\xi} - |\bchi|} \\ 
& \hspace*{60mm} \times e^{\wt(\eta)} \gch V_{ut_{\xi + \xi(u, \eta, w) + \iota(\bchi)}}^{-}(\mu). 
\end{split}
\end{equation}
\end{theorem}

\begin{proof}
By Theorems~\ref{thm:im_forgetful} and \ref{thm:inj_forgetful}, Corollary~\ref{cor:forgetful_n}, Proposition~\ref{prop:forgetful_wt}, Lemmas~\ref{lem:down_to_xi} and \ref{lem:height_to_deg}, 
and Theorem~\ref{thm:PC-type_formula}, 
we can compute as follows: 
\begin{align}
\begin{split}
& \gch V_{x}^{-}(\mu + \lambda) \\ 
&= \sum_{A \in \CA(w, \Gamma)} \sum_{\bchi \in \bPar(\lambda)} (-1)^{n(A)} q^{-\height(A) - \pair{\lambda}{\xi} - |\bchi|} e^{\wt(A)} \gch V_{\ed(A)t_{\xi + \down(A) + \iota(\bchi)}}^{-}(\mu) 
\end{split} \\ 
\begin{split}
&= \sum_{A \in \CA(w, \Gamma)} \sum_{\bchi \in \bPar(\lambda)} (-1)^{\nega(\Xi(A))) + \ell(\iota(\Xi(A))) - \ell(\iota(\Xi(A)))} q^{\Deg_{w}(\Xi(A)) - \pair{\lambda}{\xi} - |\bchi|} \\ 
& \hspace{60mm} \times e^{\wt(\Xi(A))} \gch V_{\ed(A)t_{\xi + \xi(\ed(A), \Xi(A), w) + \iota(\bchi)}}^{-}(\mu) 
\end{split} \\ 
\begin{split}
& =\sum_{\substack{\eta \in \IQLS(\lambda) \\ w \xRightarrow{(\lambda, +)} \kappa(\eta)}} \sum_{\substack{u \in W \\ \iota(\eta) \xRightarrow{(\lambda, -)} u}} \sum_{\bchi \in \bPar(\lambda)} (-1)^{\nega(\eta) + \ell(u) - \ell(\iota(\eta))} q^{\Deg_{w}(\eta) - \pair{\lambda}{\xi} - |\bchi|} \\ 
& \hspace*{60mm} \times e^{\wt(\eta)} \gch V_{ut_{\xi + \xi(u, \eta, w) + \iota(\bchi)}}^{-}(\mu), 
\end{split}
\end{align}
as desired. This proves the theorem. 
\end{proof}

\end{document}